\documentclass[12pt, a4paper]{amsart}
\usepackage{amsmath,amsfonts,amssymb,latexsym,amscd} 
\usepackage[noadjust]{cite}
\usepackage{hyperref}
\usepackage{epigraph}
\newcommand{\ddf}{\,---\,}

\usepackage[percent]{overpic}
\usepackage{color}

\theoremstyle{plain}
\theoremstyle{definition}
\newtheorem{theorem}{Theorem}
\newtheorem{lemma}{Lemma}
\newtheorem{proposition}{Proposition}
\newtheorem{definition}{Definition}
\newtheorem{corollary}{Corollary}
\newtheorem{remark}{Remark}

\newtheorem{problem}{Problem}

\def\pro-{\text{$\mathrm{pro}$-}}
\def\H-Top{\text{$\mathrm{H}$-$\mathrm{Top}$}}
\def\Sh-Top{\text{$\mathrm{Sh}$-$\mathrm{Top}$}}
\def\Set{\mathrm{Set}}
\def\Grp{\mathrm{Grp}}
\def\(H-CW){\text{$\mathrm{H}$-$\mathrm{CW}$}}

\DeclareMathOperator{\ddim}{ddim}
\DeclareMathOperator{\Fd}{Fd}
\DeclareMathOperator{\sd}{sd}
\DeclareMathOperator{\sh}{sh}
\DeclareMathOperator{\dist}{dist}
\DeclareMathOperator{\diam}{diam}

\DeclareMathOperator{\Ob}{Ob}
\DeclareMathOperator{\Mor}{Mor}

\usepackage{graphicx}
\usepackage[all]{xy}

\makeatletter
\def\@tocline#1#2#3#4#5#6#7{\relax
  \ifnum #1>\c@tocdepth % then omit
  \else
    \par \addpenalty\@secpenalty\addvspace{#2}%
    \begingroup \hyphenpenalty\@M
    \@ifempty{#4}{%
      \@tempdima\csname r@tocindent\number#1\endcsname\relax
    }{%
      \@tempdima#4\relax
    }%
    \parindent\z@ \leftskip#3\relax \advance\leftskip\@tempdima\relax
    \rightskip\@pnumwidth plus4em \parfillskip-\@pnumwidth
    #5\leavevmode\hskip-\@tempdima
      \ifcase #1
       \or\or \hskip 1em \or \hskip 2em \else \hskip 3em \fi%
      #6\nobreak\relax
    \dotfill\hbox to\@pnumwidth{\@tocpagenum{#7}}\par
    \nobreak
    \endgroup
  \fi}
\makeatother
\begin{document}

\title{Shape Theory}
\author{Pavel S. Gevorgyan}

\address{Moscow Pedagogical State University, Moscow, Russia}

\email{ps.gevorkyan@mpgu.su}
\email{pgev@yandex.ru}

\epigraph{To the blessed memory of my teacher Yurii Mikhailovich Smirnov}{}

\begin{abstract}
Shape theory was founded by K.~Borsuk 50 years ago. In essence, this is spectral homotopy theory;  
it occupies an important place in geometric topology. The article presents the basic concepts and 
the most important, in our opinion, results of shape theory. Unfortunately, many other interesting 
problems and results related to this theory could not be covered  because of space 
limitations. The article contains an extensive bibliography for those who wants to gain 
a more detailed and systematic insight into the issues considered in the survey. 

\end{abstract}

\maketitle

\tableofcontents

\section{Introduction}

Shape theory was initiated in 1968  by the renowned Polish mathematician K.~Borsuk.
In his opening speech~\cite{alekc-1979} at the 1979 International Topology Conference, 
P.~S.~Aleksandrov identified three periods in the development of geometric 
topology\footnote{According to Aleksandrov, the first of these periods was marked by 
outstanding works of Brouwer (1909--1913), Fr\'echet (1907),  and Hausdorff (1914). During the 
second period (1925--1943), homology and cohomology theory (including duality theorems) was 
developed, dimension theory was set up, and continuum theory was elaborated. 
The third period began with Borsuk's work on retract theory and continued with the 
development of shape theory (see also~\cite{alekc-fedor}).}
and mentioned shape theory as one of the most important research directions during the third period.

Shape theory is the spectral form of homotopy theory; it uses set-theoretic, geometric, and 
combinatorial-algebraic ideas and methods of topology. This theory is related to 
Aleksandrov--\v Cech homology and, therefore, to algebraic topology. 
It is also closely related to infinite-dimensional topology, in particular, to the 
theory of $Q$-manifolds.

As is known, the fundamental notions of homotopy theory can be applied to spaces with 
good local structure (such as manifolds, polyhedra, CW-complexes, and ANRs).
Many theorems of homotopy topology are valid  for CW-complexes but not valid for 
compact metrizable spaces. An example is Whitehead's celebrated theorem, which asserts  
that {\it a map $f \colon X \to Y$ between {\rm CW}-complexes is a homotopy equivalence 
if and only if it induces isomorphisms $\pi_n(f)\colon \pi_n(X) \to \pi_n(Y)$, 
$n=1,2, \dots$, of all homotopy groups}. This theorem is not generally true for 
compact metrizable spaces. Indeed, consider the so-called \emph{Warsaw circle}, 
or \emph{quasi-circle}, $W$, which consists of the graph of the function 
$y=\sin {\dfrac{2\pi}x}$, $0<x\leqslant 1$, the interval $[-1,1]$ of the $y$-axis, 
and  the arc joining the points $(0,0)$ and $(1,0)$ (see Fig.~\ref{ris_1}). 
It has bad local structure and is not a CW-complex.
\begin{figure}[ht]\label{ris_1}
\begin{center}
\begin{overpic}[scale=0.6]{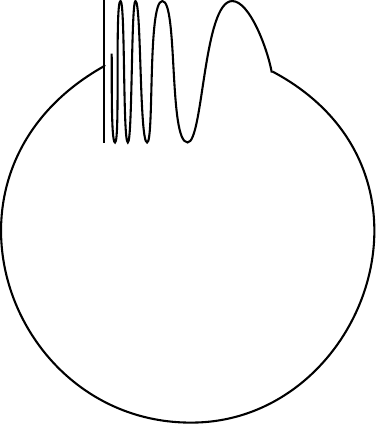}
\put(69,83){$W$}
\end{overpic}
\end{center}
\quad \ Fig. 1
 \label{wc}
\end{figure}
Any continuous map of the $n$-sphere $S^n$ to the Warsaw circle $W$ is homotopic to a constant map, 
and hence all homotopy groups of the Warsaw circle are trivial. However, $W$ does not have the 
homotopy type of a point. Thus, the constant map $f \colon W \to \{*\}$ to a one-point space 
induces isomorphisms of the homotopy groups for all $n$ but is not a homotopy equivalence.

Homotopy theory is ideally suited for being applied to spaces with good local 
structure but not to arbitrary spaces. The Warsaw circle $W$ and 
the usual circle $S^1$ are not homotopy equivalent, although they look similar 
and share certain important properties. For example, each of them separates the plane into
two parts.

The reason why the Warsaw circle and the circle $S^1$ are not homotopy equivalent is that there 
are not enough continuous maps from $S^1$ to $W$. 
The situation changes when continuous maps from $S^1$ to arbitrary neighborhoods 
of the Warsaw circle $W$ are considered. Such maps are already sufficiently many. This 
observation has enabled Borsuk to propose a remarkable idea of addressing this 
``drawback'' of homotopy theory for compact metrizable spaces and, thereby, open up a 
new direction of geometric topology, shape  theory \cite{borsuk-1}. 
Borsuk first reported his results at the 1967 Symposium on Infinite-Dimensional Topology, 
which was held in Baton Rouge, Louisiana,  on March 27--April 1, 1967, and then at the 
International Symposium on Topology and Its Applications in Herceg 
Novi on August 25--31, 1968 (see~\cite{borsuk-2}).

After the appearance of Borsuk's pioneering papers \cite{borsuk-6}--\cite{borsuk-5}, 
shape theory began to rapidly develop across the world. Warsaw, where Borsuk and his students 
worked, had become the center of scientific activities and research on shape theory. In the 
United States the first studies on shape theory were performed by R.~Fox, T.~A.~Chapman, 
J.~Segal, and J.~Keesling, and in Japan, by K.~Morita and Y.~Kodama. In Moscow research in 
shape theory was headed by Yu.~M.~Smirnov. In Zagreb a group of topologists working on this 
theory was formed; it was led by S.~Marde\v si\'c. In Frankfurt shape theory was developed 
by F.~Bauer and his students, and in Great Britain, by T.~Porter; many other 
topologists worked on it all over the world.

Borsuk's idea underlying shape theory is based on the fact that any compact metrizable space embeds 
in the Hilbert cube $Q$ (as well as in any other absolute retract) and consists in considering, 
instead of continuous maps between compact metrizable spaces $X$ and $Y$, so-called fundamental 
sequences $(f_{n}):Q\to Q$ coordinated in a certain way with neighborhoods of $X$ and $Y$ (see 
Sec.~\ref{osn_constr}).

In 1972 Fox \cite{fox} extended shape theory to metrizable spaces by using Borsuk's method. To this 
end, he embedded an arbitrary metric space $X$ in an absolute 
neighborhood retract $M$ as a closed subspace and constructed shape morphisms by using a 
system of neighborhoods of the space $X$ in $M$. Importantly, all these neighborhoods are 
absolute neighborhood retracts for metrizable spaces.

Borsuk's geometric method does not apply in the general case, because there are too few absolute 
neighborhood retracts in the category of all topological spaces.

In 1970 Marde\v{s}i\'{c} and Segal \cite{mard-seg-2}, \cite{mard-seg-3} constructed shape theory 
for arbitrary compact Hausdorff spaces by means of inverse systems. They essentially used the 
well-known theorem that \emph{any compact Hausdorff set $X$ is the limit of an inverse system 
${\bf X}=\{X_\alpha, p_{\alpha \alpha'}, A\}$ of compact ANR spaces 
$X_\alpha$}.

In 1975 Morita \cite{morita-1}  extended 
shape theory to  arbitrary topological spaces by the method of inverse ANR-systems. 
He had succeeded thanks to  \v Cech's functor 
$\check{C}$; to each topological space $X$ this functor assigns an inverse system $\check{C}(X)$ 
of the nerves $X_\alpha$ of normal locally finite open covers $\alpha$ with the natural projections 
generated by the refinement of covers. Morita introduced the notion of the \emph{associated inverse 
ANR-system} and proved that, for any topological space $X$, the inverse ANR-system 
$\check{C}(X)$ is associated with $X$. Moreover, all inverse ANR-systems associated with a given 
space $X$ are isomorphic to each other and, in particular, to the system $\check{C}(X)$.

The spectral method for constructing shape theory has turned out to be universal. Thus, 
in \cite{borsuk-1}, given a compact metrizable space $X$,  Borsuk essentially used  
an associated countable ANR-system of its neighborhoods which arises under an embedding of 
$X$ in an absolute retract $M$. Moreover, since the ANR-system obtained from the system of all 
neighborhoods of a closed subset $X$ of an absolute neighborhood retract $N$ in the class 
of metrizable spaces is associated with $X$, we see that Fox \cite{fox} essentially considered 
inverse ANR-systems associated with metrizable spaces. Finally, if a compact Hausdorff 
space $X$ is the limit of an inverse system ${\bf X}=\{X_\alpha, p_{\alpha \alpha'}, A\}$ 
of compact $\mathrm{ANR}$-spaces $X_\alpha$, i.e., 
$X=\lim\limits_{\longleftarrow}\{X_\alpha, p_{\alpha \alpha'}, A\}$, then the ANR-system  
$\{X_\alpha, \mathbf{p}_{\alpha \alpha'}, A\}$ in the homotopy category H-CW is associated with 
$X$. This has enabled Marde\v{s}i\'{c} and Segal \cite{mard-seg-2}, 
\cite{mard-seg-3} to construct shape theory for arbitrary compact Hausdorff spaces.

In the 1973 paper \cite{mard-1}  Marde\v{s}i\'{c} gave a short categorical definition of 
shape theory. The objects of the shape category Sh-Top are topological spaces, and morphism 
$F\colon X\to Y$ are maps which take each homotopy class $\boldsymbol{\psi}\colon Y\to P$, where 
$P$ is an ANR, to a homotopy class $\boldsymbol{\varphi}\colon X\to P$ 
such that, given any ANR $P'$ and any homotopy classes $\boldsymbol{\psi'}\colon Y\to P'$ and 
$\mathbf{q}\colon P'\to P$ satisfying the relation 
$\mathbf{q}\boldsymbol{\psi'}=\boldsymbol{\psi}$, we have 
$\mathbf{q}\boldsymbol{\varphi'}=\boldsymbol{\varphi}$.

The shape functor $S\colon \H-Top\to \Sh-Top$ satisfies the following two conditions: 

(i)~$S(X)=X$ 
for any topological space $X$ and 

(ii)~for any shape morphism $F\colon X\to Q$ to an ANR 
$Q$, there exists a unique homotopy class $\mathbf{f}\colon X\to Q$ such that $S(\mathbf{f})=F$. 

These conditions have the universality property,  and it is natural to call them the \emph{axioms 
of shape theory}. The first axiomatics of shape theory for the class of all compact metrizable  
spaces was proposed by Holszty\'{n}ski \cite{hol}. In the general case, an axiomatic 
characterization of shape theory was given by Marde\v{s}i\'{c}~\cite{mard-1}.

For compact metrizable spaces, all methods for constructing shape theory are equivalent
to Borsuk's. However, in the class of metrizable spaces, the shape theories of Borsuk 
\cite{borsuk-8}, Fox \cite{fox}, and Bauer \cite{bauer} differ from each other. Thus, there 
exist different shape theories (likewise, in algebraic topology, there are 
different homology and cohomology theories).

The shape category is weaker than the homotopy category. The former is not a generalization of the 
latter; it is rather a natural correction, or extension, of the homotopy category, because the shape 
functor $S\colon \H-Top\to \Sh-Top$ is an isomorphism onto the subcategory H-CW of all spaces 
with the homotopy type of an ANR. Thus, for the class of spaces having the homotopy type 
of a polyhedron, which includes all CW-complexes and ANRs, shape theory coincides 
with homotopy theory. The Warsaw circle and the circle $S^1$ are examples of compact metrizable 
spaces which have different homotopy types but are shape equivalent.

In essence, shape theory is spectral homotopy theory, because it is based on the idea 
of replacing the singular by the approximative; in this respect, it is closely 
related to Aleksandrov--\v Cech spectral homology. Shape theory is intended for application 
to spaces with bad local structure, which increasingly frequently arise 
in very diverse areas of mathematics; they also result 
from applying many topological constructions, such as bundles, cell-like maps, 
fixed point sets, attractors of dynamical systems, spectra of linear operators, 
remainders of compactifications, boundaries of groups, etc.

Ideas and problems closely related to shape theory have been known long before the 
publication of Borsuk's first papers \cite{borsuk-1} and \cite{borsuk-2}, in which  
the shape category for compact metrizable spaces was constructed. As early as in 1895 
Poincar\'e  laid the foundations of algebraic topology and introduced the important notions of  
a chain, a cycle, a boundary, the Betti numbers, homologous cycles, etc.\ in his celebrated treatise 
\emph{Analysis Situs} \cite{poincare-1}. Poincar\'e also conjectured (but not proved)  that 
the Betti numbers are topological invariants. The first rigorous proof of this fact was given 
by Alexander  in 1915 \cite{alexander}, which had put Poincar\'e's intuitive ideas on 
solid mathematical grounds. In 1922 Veblen \cite{veblen} proved the topological invariance 
of singular homology. The application of singular homology is based on the consideration of a 
family of continuous maps from a polyhedron $P$ to a topological space $X$. Homotopy theory is 
based on the same idea, because each homotopy group $\pi_n(X, *)$ 
is defined by using continuous maps from the $n$-sphere $S^n$ to a pointed space $(X,*)$. However, 
this idea does not work for spaces with bad local structure. This is why shape theory is based on 
the dual idea of considering families of continuous maps from a topological space $X$ to a 
polyhedron $P$. This approach was first applied by Aleksandrov to define inverse systems of 
spaces~\cite{aleks-1}, \cite{aleks-2} and introduce the notion of the nerve of a cover of a 
topological space \cite{aleks-3}. Thus, the introduction of \v Cech homology in various forms by 
Aleksandrov \cite{aleks-1}, \cite{aleks-2}, Vietoris \cite{vietoris}, and \v Cech \cite{cech} can 
be justly regarded as the rudiment of shape theory.

Shape theory is related to many fields of topology, and its ideas and results turn out to be 
important and useful for these fields. Because of space limitations, we could not cover all 
interesting problems and remarkable results of shape theory. Nevertheless, we  
have tried to touch upon the most important points in the development of shape theory 
during the past five decades.

The literature on shape theory is fairly extensive (see the bibliography at the end of this 
paper); we mention Borsuk's monograph \cite{borsuk-rus}, the books 
\cite{mard-seg-4} by  Marde\v si\'c and Segal and \cite{dydak-segal-1} by Dydak and Segal, 
and survey papers \cite{smirnov-teorsh} by  Smirnov, \cite{borsuk-dydak} by Borsuk and Dydak, 
 and \cite{mard-absolute} and \cite{mard-thirty} by Marde\v si\'c.

\section{Basic Constructions of Shape Theory}\label{osn_constr}

\subsection{The Shape Category of Compact Metrizable Spaces (Borsuk's Method)}

The main idea of Borsuk's construction of the shape category of compact metrizable spaces is 
the replacement of the class of continuous maps by a larger class of morphisms and 
the introduction of the notion of a homotopy between such morphisms.

Let $X$ and $Y$ be compact metrizable spaces contained in ARs $M$ and $N$, 
respectively.

\begin{definition}\label{def-1}
A sequence of maps $f_n\colon M\to N$, $n\in \mathbb{N}$, is called a \emph{fundamental sequence 
from $X$ to $Y$} and denoted  by $(f_n)_{MN}\colon X\to Y$ if, for any neighborhood $V$ of 
the compact set $Y$ in $N$, there exists a neighborhood $U$ of the compact set $X$ in $M$ and a 
positive integer $n_V\in \mathbb{N}$ such that, first, $f_n(U)\subset V$ for all $n\geqslant n_V$ 
and, secondly, % 
\[
 f_n|U \simeq f_m|U
\]
in $V$ for all $n,m\geqslant n_V$.
\end{definition}

The fundamental sequence $(1_n)_{MM}\colon X\to X$, where $1_n=1_M\colon M\to M$ is the identity 
map for all $n\in \mathbb{N}$, is called the \emph{identity fundamental sequence}. The composition 
of fundamental sequences $(f_n)_{MN}\colon X\to Y$ and $(g_n)_{NP}\colon Y\to Z$, where $Z$ 
is a compact metrizable space contained in an AR $P$, is defined by % 
\[ 
(g_n)(f_n)=(g_nf_n). 
\] % 
Clearly, $(g_nf_n)_{MP}\colon X\to Z$ is a fundamental sequence, and 
the composition operation is associative.

The notion of homotopy naturally extends to fundamental sequences.

\begin{definition}
Two fundamental sequences $(f_n)_{MN}\colon X\to Y$ and $(f'_n)_{MN}\colon X\to Y$ are said to be 
\emph{homotopic} if, for any neighborhood $V$ of the compact set $Y$ in $N$, there exists a 
neighborhood $U$ of the compact set $X$ in $M$ and a positive integer $n_V\in \mathbb{N}$ such that %
\[ 
f_n|U \simeq f'_n|U 
\] 
in $V$ for all $n\geqslant n_V$. In this case, we write $(f_n)\simeq (f'_n)$. 
\end{definition}

This relation is reflexive, symmetric, and transitive. The equivalence class of a fundamental 
sequence $(f_n)_{MN}\colon X\to Y$ is called a \emph{fundamental class} and denoted by $[(f_n)]$.  
It is easy to see that if $(f_n)\simeq (f'_n)$ and $(g_n)\simeq (g'_n)$, then $g_nf_n \simeq 
g'_nf'_n$. Therefore, we can define the composition of fundamental classes $[(f_n)]$ and $[(g_n)]$ 
as 
\[ 
[(g_n)][(f_n)]=[(g_n)(f_n)]. 
\] 
Obviously, this composition operation is associative as well, and 
$$[(f_n)][(1_n)]=[(1_n)][(f_n)]=[(f_n)].$$

Thus, we obtain the so-called \emph{fundamental category}, whose objects are all pairs  
$(X,M)$, where $M$ is an AR and $X$ is a compact metrizable subspace of $M$, and 
morphisms are classes of fundamental sequences.  Isomorphic objects in this category are said to be 
\emph{fundamentally equivalent}. In other words, pairs $(X,M)$ and $(Y,N)$ are fundamentally 
equivalent if there exist fundamental sequences  $(f_n)_{MN}\colon X\to Y$ and 
$(g_n)_{NM}\colon Y\to X$ such that $[(g_n)][(f_n)]=[(1_n)_{MM}]$ and  
$[(f_n)][(g_n)]=[(1_n)_{NN}]$.

We say that a fundamental sequence  $(f_n)_{MN}\colon X\to Y$ is \emph{generated by a map 
$f\colon X\to Y$} if $f_n|X=f$ for all $n \in \mathbb{N}$. Clearly, any map $f\colon  X \to Y$ 
generates a fundamental sequence $(f_n)_{MN}\colon X\to Y$. Indeed, since $N$ is 
an AR, it follows that there exists an extension $\tilde{f}\colon M\to N$ of the map $f\colon X\to 
Y$. Setting $f_n=\tilde{f}$ for each $n=1,2, \dots$, we obtain a fundamental sequence 
generated by~$f$.

The following theorem shows that the fundamental class of a fundamental sequence $(f_n)_{MN}\colon 
X\to Y$ generated by a map  $f\colon  X \to Y$ depends only on the homotopy class of this map.

\begin{theorem}\label{th_1}
Suppose that fundamental sequences $(f_n)_{MN}\colon X\to Y$ and $(f'_n)_{MN}\colon X\to Y$ 
are generated by maps $f\colon X\to Y$ and $f'\colon X\to Y$, respectively, and $f\simeq 
f'$. Then $(f_n)\simeq (f'_n)$. 
\end{theorem}

\begin{proof}
Consider any neighborhood $V$ of the compact set $Y$ in $N$. By Definition~\ref{def-1} there 
exists a neighborhood $U_1$ of the compact set $X$ in $M$ and a positive integer $n'_V \in 
\mathbb{N}$ such that $f_n|U_1 \simeq f_m|U_1$ in $V$ for all $n,m\geqslant n'_V$. 
Similarly, there exists a neighborhood $U_2$ of the compact set $X$ in $M$ 
and a positive integer $n''_V \in \mathbb{N}$ such that $f'_n|U_2 \simeq f'_m|U_2$ in $V$ for all 
$n,m\geqslant n''_V$. We set $\widetilde{U}=U_1\cap U_2$ and $n_V=\max\{n'_V,n''_V\}$. Clearly, 
\begin{equation}\label{eq-1} f_n|\widetilde{U} \simeq f_m|\widetilde{U} 
\quad \text{and} \quad 
f'_n|\widetilde{U} \simeq f'_m|\widetilde{U} 
\end{equation} 
in $V$ for all $n,m\geqslant n_V$.

Now consider the maps $f_{n_V}, f'_{n_V}\colon \widetilde{U} \to V$. By the assumption of 
the theorem, we have $f_{n_V}|X=f$, $f'_{n_V}|X=f'$, and $f\simeq f'$. Let $F\colon X\times I \to 
Y$ be a homotopy between $f$ and $f'$. Since $V$ is an ANR, it follows that there exists 
a neighborhood $U$ of $X$ in $M$ such that $U\subset \widetilde{U}$ and a homotopy 
$\widetilde{F}\colon U\times I \to V$ between $f_{n_V}|U$ and $f'_{n_V}|U$ such that 
$\widetilde{F}|X \times I = F$ (see \cite{mard-seg-4}, Theorem 8, p.~40). Thus, 
\begin{equation}
\label{eq-2} f_{n_V}|U \simeq f'_{n_V}|U. 
\end{equation}

It remains to note that
\[
f_n|U \simeq f'_n|U
\]
in $V$ for all $n\geqslant n_V$; this follows from \eqref{eq-1} and~\eqref{eq-2}.
\end{proof}

\begin{corollary}\label{cor_1}
Let $(f_n)_{MN}\colon X\to Y$ and $(f'_n)_{MN}\colon X\to Y$ be fundamental sequences generated 
by a map $f\colon X\to Y$. Then  $(f_n)\simeq (f'_n)$. 
\end{corollary}

This corollary implies, in particular,  that all fundamental sequences $(i_n)_{MM}\colon 
X\to X$ generated by the map $1_X\colon X\to X$ are homotopic to the identity fundamental sequence 
$(1_n)_{MM}\colon X\to X$.

The following important theorem shows that the fundamental equivalence relation is absolute 
in the sense that it does not depend on the choice of the ARs in which 
metrizable compact spaces are embedded.

\begin{theorem}
Let $X\subset M\cap M'$. Then the pairs $(X,M)$ and $(X,M')$ are fundamentally equivalent.
\end{theorem}

\begin{proof}
Consider fundamental sequences $(i_n)_{MM'}\colon X\to X$ and $(i_n)_{M'M}\colon X\to X$ generated 
by the identity map $1_X\colon X\to X$. Obviously, the compositions $(i_n)_{M'M}(i_n)_{MM'}$ 
and $(i_n)_{MM'}(i_n)_{M'M}$ are generated by the identity map $1_X\colon X\to X$ as well; 
hence, according to Corollary~\ref{cor_1}, they are homotopic to the identity fundamental sequences 
$(1_n)_{MM}$ and $(1_n)_{M'M'}$, respectively. Thus, $[(i_n)_{M'M}][(i_n)_{MM'}]=[(1_n)_{MM}]$ and 
$[(i_n)_{MM'}][(i_n)_{M'M}]=[(1_n)_{M'M'}]$, i.e.,  $(X,M)$ and $(X,M')$ are fundamentally 
equivalent. 
\end{proof}

\begin{remark}
Since each metrizable compact space $X$ is homeomorphic to a compact subset of the 
Hilbert cube $Q$, in constructing the fundamental category, is suffices to consider pairs 
of the form $(X,Q)$. 
\end{remark}

The category whose objects are compact metrizable spaces and morphisms are classes 
of fundamental sequences between them is called the \emph{shape category} of compact metrizable 
spaces. We denote it by Sh(CM). Since the fundamental equivalence relation is an 
equivalence relation, it follows that the class of all compact metrizable spaces splits 
into pairwise disjoint classes of spaces, which are called \emph{shapes}. 
The shape containing a space $X$ is called the \emph{shape of the space} $X$ and denoted by 
sh($X$). We say that two compact spaces $X$ and $Y$ \emph{are shape equivalent}, 
or \emph{have the same shape}, and write $\sh (X)= \sh (Y)$, if they are fundamentally equivalent.

%---------------------------------

\subsection{The Spectral Method for Constructing Shape Theory (the Marde\v{s}i\'c--Morita Method  \cite{mard-1}, \cite{morita-1})}

To define the shape category of arbitrary topological spaces, we need the pro-homotopy 
category pro-N-TOP, which was introduced by Grothendieck in \cite{groth}. The objects of this 
category are inverse systems $\mathbf{X}=\{ X_\alpha , \mathbf{p}_{\alpha\alpha '}, A \}$ 
in the homotopy category  H-Top of topological spaces. The morphisms of the category pro-N-TOP are 
defined in two steps. First, a morphism $(\mathbf{f}_\beta, \varphi) \colon  
\mathbf{X}\to \mathbf{Y}=\{Y_\beta, \mathbf{q}_{\beta\beta'}, B\}$ between inverse systems is 
defined as a pair of maps $\varphi \colon B \to A$ and $\mathbf{f}_\beta \colon X_{\varphi(\beta)} 
\to Y_\beta$ satisfying the following condition:

(*) for any $\beta, \beta'\in B$, $\beta' \geqslant \beta$, there exists an index 
$\alpha \geqslant \varphi(\beta), \varphi(\beta')$ such that $\mathbf{f}_\beta 
\mathbf{p}_{\varphi(\beta)\alpha} = \mathbf{q}_{\beta\beta'} \mathbf{f}_{\beta'} 
\mathbf{p}_{\varphi(\beta')\alpha}$, i.e., the following diagram is commutative: 
$$ \xymatrix { & 
{X_\alpha} \ar[dl]_{\mathbf{p}_{\varphi(\beta)\alpha}}   
\ar[dr]^{\mathbf{p}_{\varphi(\beta')\alpha}} &
\\ 
X_{\varphi(\beta)}  \ar[d]_{\mathbf{f}_\beta} & & X_{\varphi(\beta')} \ar[d]^{\mathbf{f}_{\beta'}} 
\\ Y_{\beta}   & & Y_{\beta'} 
\ar[ll]^{\mathbf{q}_{\beta\beta'}} } 
$$

Then, the  notion of equivalent morphisms is introduced.
Morphisms $(\mathbf{f}_\beta, \varphi)$ and $(\mathbf{g}_\beta, \psi)$ from an inverse 
system $\mathbf{X}=\{ X_\alpha , \mathbf{p}_{\alpha\alpha '}, A \}$ to an inverse system 
$\mathbf{Y}=\{Y_\beta, \mathbf{q}_{\beta\beta'}, B\}$ are said to be \emph{equivalent} if the 
following condition holds:

(**) for any $\beta \in B$, there exists an index $\alpha \geqslant \varphi(\beta), 
\psi(\beta)$ such that $\mathbf{f}_\beta \mathbf{p}_{\varphi(\beta)\alpha} = \mathbf{g}_{\beta} 
\mathbf{p}_{\psi(\beta)\alpha}$, i.e.,  the following diagram is commutative: 
$$ \xymatrix { & 
{X_\alpha} \ar[dl]_{\mathbf{p}_{\varphi(\beta)\alpha}}   \ar[dr]^{\mathbf{p}_{\psi(\beta)\alpha}}&
\\ 
X_{\varphi(\beta)}  \ar[dr]_{\mathbf{f}_\beta} & & X_{\psi(\beta)} \ar[dl]^{\mathbf{g}_{\beta}} 
\\ 
& Y_{\beta} & } 
$$ 
This relation is indeed an equivalence relation. Thus, the set of all morphisms 
from an inverse system $\mathbf{X}=\{ X_\alpha , \mathbf{p}_{\alpha\alpha '}, A \}$ to an 
inverse system $\mathbf{Y}=\{Y_\beta, \mathbf{q}_{\beta\beta'}, B\}$ splits into 
equivalence classes $[(\mathbf{f}_\beta, \varphi)]$; it is these classes which are morphisms 
$\mathbf{f} \colon\mathbf{X} \to \mathbf{Y}$ of the pro-homotopy category pro-N-TOP. 
The composition of two morphisms $\mathbf{f} \colon\mathbf{X} \to \mathbf{Y}$ and $\mathbf{g} 
\colon\mathbf{Y} \to \mathbf{Z}=\{ Z_\gamma , \mathbf{r}_{\gamma\gamma '}, \Gamma \}$ is defined 
by using their representatives $\mathbf{f}=[(\mathbf{f}_\beta, \varphi)]$ and 
$\mathbf{g}=[(\mathbf{g}_\gamma, \psi)]$ as $\mathbf{g}\circ 
\mathbf{f}=[(\mathbf{g}_\gamma\circ \mathbf{f}_{\psi(\gamma)}, \varphi\circ \psi)]$.

\begin{remark}
The morphisms between inverse systems 
$\mathbf{X}=\{ X_\alpha , \mathbf{p}_{\alpha\alpha '}, A \}$ and 
$\mathbf{Y}=\{Y_\beta, \mathbf{q}_{\beta\beta'}, B\}$ in the pro-homotopy category pro-N-TOP are   
described by Grothendieck's formula \cite{groth} 
\[ 
\Mor(\mathbf{X}, \mathbf{Y})=\lim_{\underset{\beta}{\longleftarrow}} 
\lim_{\underset{\alpha}{\longrightarrow}}[X_\alpha, Y_\beta], 
\] 
where the $[X_\alpha, Y_\beta]$ are homotopy classes of maps. 
\end{remark}

In constructing shape theory, of special interest is the pro-homotopy category pro-H-CW of inverse 
systems in the complete subcategory H-CW of the category H-Top. The following well-known result 
gives an important characterization of spaces in the category H-CW (see \cite{mard-seg-4}, Chap.~I, 
Sec.~4.1, Theorem 1).

\begin{lemma}
For a topological space $X$, the following conditions are equivalent:

(a) $X$ has the homotopy type of a CW-complex;

(b) $X$ has the homotopy type of a simplicial complex with the metric topology;

(c) $X$ has the homotopy type of an ANR (the class of metrizable spaces).
\end{lemma}

Taking into account this lemma, we can say that H-CW is the homotopy category of spaces having 
the homotopy type of an ANR. In what follows, we shall adhere to this approach, because in 
shape theory only properties of ANRs are used; we refer to inverse systems in the category 
pro-H-CW as \emph{ANR-systems}.

An important role in the construction of  shape theory for the class of all topological spaces 
is played by  the notion of an associated ANR-system, which was introduced by Morita 
in~\cite{morita-1}.

\begin{definition}[Morita \cite{morita-1}]\label{def-3}
An inverse ANR-system $\mathbf{X}=\{ X_\alpha , \mathbf{p}_{\alpha\alpha '}, A \}$ is said to be 
\emph{associated with a topological space $X$} if there exist 
homotopy classes $\mathbf{p}_\alpha \colon  X\to X_\alpha$, $\alpha \in A$, satisfying the following 
conditions:

(i) $\mathbf{p}_\alpha = \mathbf{p}_{\alpha\alpha'} \mathbf{p}_{\alpha'}$ 
for all $\alpha \leqslant \alpha'$;

(ii) for any ANR-space $P$ and any homotopy class $\mathbf{f}\colon X\to P$, there exists an index  
$\alpha\in A$ and a homotopy class $\mathbf{h}_\alpha\colon X_\alpha \to P$ such that 
$\mathbf{h}_\alpha \mathbf{p}_\alpha = \mathbf{f}$;

(iii) if $\boldsymbol{\varphi} \mathbf{p}_\alpha = \boldsymbol{\psi} \mathbf{p}_\alpha$ for two  
homotopy classes $\boldsymbol{\varphi}, \boldsymbol{\psi} \colon X_\alpha \to P$, then there exists 
an index $\alpha'\geqslant \alpha$ such that $\boldsymbol{\varphi} 
\mathbf{p}_{\alpha\alpha'} = \boldsymbol{\psi} \mathbf{p}_{\alpha\alpha'}$. 
\end{definition}

\begin{theorem}[Morita \cite{morita-1}]
For any topological space $X$, there exists an associated inverse ANR-system.
\end{theorem}

\begin{proof}
Let $\mathcal{U}_{\alpha}, \alpha \in A$, be all normal locally finite open covers of the space 
$X$. If a cover $\mathcal{U}_{\alpha '}$ is a refinement of $\mathcal{U}_{\alpha}$, 
then we write $\alpha ' > \alpha$. Let $X_\alpha$ be the nerve of $\mathcal{U}_{\alpha}$; 
this is a simplicial complex with the weak topology. For each $\alpha\in A$, consider 
a continuous map $p_\alpha\colon  X\to X_\alpha$ satisfying the condition
\[
p_\alpha^{-1}(St(u,X_\alpha))\subset U,
\]
where $u$ is the vertex of $X_\alpha$ corresponding to the element $U$ of the cover 
$\mathcal{U}_{\alpha}$. Such a map is said to be \emph{canonical}. All canonical maps are homotopic 
to each other. If $\alpha ' > \alpha$, then there exists a simplicial map $p_{\alpha\alpha 
'}\colon X_{\alpha '}\to X_\alpha $ such that $p_{\alpha \alpha '}(u)=v$ implies 
$U\subset V$, where $u$ and $v$ are the vertices of the simplicial complexes $X_{\alpha}$ and 
$X_{\alpha '}$ corresponding to the elements $U$ and $V$ of the covers $\mathcal{U}_{\alpha}$ and 
$\mathcal{U}_{\alpha '}$, respectively. The map $p_{\alpha \alpha '}$ is called a \emph{canonical 
projection}. Any two canonical projections are homotopic to each other. Moreover, 
\[ 
\mathbf{p}_{\alpha \alpha '} \mathbf{p}_{\alpha '} = \mathbf{p}_{\alpha} 
\] 
for all $\alpha, \alpha '\in A$, $\alpha > \alpha '$.

The inverse ANR-system $\{X_\alpha, \mathbf{p}_{\alpha \alpha '}, A\}$ thus obtained is associated 
with the space~$X$. 
\end{proof}

Inverse ANR-systems associated with different spaces can be constructed by different methods. For 
example, any compact Hausdorff space $X$ is the limit of an inverse system $\{X_\alpha, 
p_{\alpha \alpha '}, A\}$ consisting of compact ANRs $X_\alpha$: 
$X=\lim\limits_{\longleftarrow}\{X_\alpha, p_{\alpha \alpha '}, A\}$. It turns out that the 
inverse ANR-system $\{X_\alpha, \mathbf{p}_{\alpha \alpha '}, A\}$ is associated with~$X$.  
However, importantly, all  inverse ANR-systems associated with the same space $X$ are equivalent in 
the category  pro-H-CW.

Thus, each topological space $X$ is assigned an inverse ANR-system $S(X)$ in the category pro-H-CW. 
Moreover, this assignment can be uniquely extended in a natural way to 
morphisms of the category H-Top, i.e., each homotopy class  $\mathbf{f}\colon X\to Y$ can be assigned 
a morphism $S(\mathbf{f})\colon S(X)\to S(Y)$ of the category pro-H-CW. Thereby, we obtain 
a functor S: H-Top $\to$ pro-H-CW, which is called the \emph{shape functor}.

The shape functor $S$ makes it possible to naturally define the \emph{shape category} Sh-Top for 
the class of all topological spaces. The objects of the category Sh-Top are those of the 
category H-Top (i.e., these are topological spaces), and the morphisms are those of the category 
pro-H-CW: $S(X,Y)=\text{MOR}(S(X),S(Y))$. The shape classification of spaces is weaker 
than the homotopy classification, but on the class H-CW, both classifications coincide. In the class 
of all compact metrizable spaces, Marde\v si\'c--Morita shape theory  coincides with Borsuk's one.

Shape theory can be constructed without employing associated inverse systems as follows
(see Marde\v{s}i\'{c} \cite{mard-1}). To each topological space 
$X$ we assign a category $W^X$ whose  objects are homotopy classes $\mathbf{f} \colon  X \to P$, $P\in \text{H-CW}$, and morphisms $\mathbf{u}\colon \mathbf{f} \to \mathbf{f'}$, where $\mathbf{f'} \colon  X \to P'$, $P'\in \text{H-CW}$, are homotopy classes $\mathbf{u}\colon  
P\to P'$ such that $\mathbf{u} \mathbf{f} = \mathbf{f'}$. Note that the identity morphism 
$1_\mathbf{f}\colon \mathbf{f} \to \mathbf{f}$ is the homotopy class 
$\mathbf{1}_P\colon P\to P$, and the compositions of morphisms in the category $W^X$ are 
the compositions of the corresponding homotopy classes in the category H-CW. Now the shape 
morphisms $F\colon X\to Y$ are defined as the covariant functors $F \colon  W^Y \to W^X$ 
satisfying the following two conditions: (i)~if $\mathbf{g}\in W^Y$, $\mathbf{g}\colon Y \to P$, 
then $F(\mathbf{g}) =\mathbf{f}\in W^X$, where $\mathbf{f}\colon X \to P$; (ii)~if 
$\mathbf{u}\colon \mathbf{g} \to \mathbf{g'}$ is the morphism determined by a homotopy class 
$\mathbf{u}\colon P\to P'$,  then the morphism $F(\mathbf{u}) \colon  F(\mathbf{g}) \to 
F(\mathbf{g'})$ is determined by the same homotopy class $\mathbf{u}\colon  P\to P'$.

By using this definition of the shape category, is easy to prove the following criterion for a  
continuous map to be a shape equivalence.

\begin{theorem}
A continuous map $f\colon X\to Y$ is a shape equivalence if and only if, for any ANR $P$, 
the induced map $f^{*}\colon [Y,P]\to [X,P]$ is one-to-one. 
\end{theorem}

Here $[Y, P]$ denotes the set of homotopy classes $\mathbf{g}\colon  Y\to P$ and $f^*$ is the 
map assigning the homotopy class $f^*(\mathbf{g})=\mathbf{g} \mathbf{f} \in [X,P]$ to each 
$\mathbf{g}\colon  Y\to P$.

\subsection{Equivariant Shape Theory}

There are many ways to construct shape category for $G$-spaces (continuous transformation groups 
with a fixed action of a group $G$). Smirnov 
\cite{smirnov-g-par}, \cite{smirnov-1979} 
constructed shape theory for the categories of metrizable $G$-spaces  and compact $G$-spaces in the 
case of a compact acting group $G$ by the Borsuk--Fox method \cite{borsuk-rus}, \cite{fox}.  
Equivariant shape theory for arbitrary $G$-spaces in the case of a finite group $G$ was constructed 
independently by Pop \cite{pop-shape} and Matumoto \cite{matumoto} and in the case of a compact 
group $G$, by Antonyan and Marde\v{s}i\'{c} \cite{antonyan-mard} and 
Gevorgyan \cite{gev-smirnov}. \v{C}erin \cite{cerin} constructed the shape category for arbitrary 
$G$-spaces in the case of any acting group $G$ by using $G$-homotopy classes of families 
of set-valued maps.

In \cite{gev-smirnov} equivariant shape theory was constructed by a method based on the application 
of all invariant continuous pseudometrics on a given $G$-space. Let $X$ be any $G$-space, and 
let $\mu$ be an invariant continuous pseudometric on $X$. Consider the equivalence relation on $X$ 
defined by setting $x\sim x'$ if and only if $\mu(x,x')=0$. We denote the corresponding quotient 
space $X|_\sim$ by $X_{\mu}$ and the equivalence class of an element $x\in X$ by $[x]_{\mu}$. 
The quotient space $X_{\mu}$ is a $G$-space with action $g[x]_{\mu}=[gx]_{\mu}$, and $\rho 
([x]_{\mu},[x']_{\mu})=\mu(x,x')$ is an invariant metric on $X_{\mu}$. Note also that 
the quotient map $p_\mu\colon X\to X_\mu$ defined by $p_\mu(x)=[x]_{\mu}$ is 
continuous and equivariant. Thus, the following assertion is valid.

\begin{proposition}
The quotient space $X_\mu$ is an invariant metrizable $G$-space,
and the quotient map $p_\mu\colon X\to X_\mu$ is an equivariant continuous map.
\end{proposition}

The following theorem makes it possible to construct equivariant shape theory (see Gevorgyan's 
paper \cite{gev-smirnov}) for the class of all $G$-spaces in the case a compact group $G$ by the 
Marde\v si\'c--Morita method of inverse systems.

\begin{theorem}\label{th-main}
For any $G$-space $X$, there exists an associated inverse system $\{X_\alpha, 
\mathbf{p}_{\alpha\alpha'}, A\}$ in the equivariant homotopy category H-G-ANR. 
\end{theorem}

\begin{proof}
Consider the family $\mathcal{P}$ of all invariant pseudometrics on the $G$-space $X$.
We introduce a natural order on this family by setting $\mu\leqslant \mu'$ if 
$\mu(x,x')\leqslant \mu'(x,x')$ for $x,x'\in X$.

Given any pseudometrics $\mu, \mu'\in \mathcal{P}$, $\mu\leqslant \mu'$, we define a map 
$p_{\mu\mu'}\colon X_{\mu'}\to X_\mu$ by $p_{\mu\mu'}([x]_{\mu'})=[x]_{\mu}$, where 
$[x]_{\mu'}\in X_{\mu'}$ and $[x]_{\mu}\in X_{\mu}$. This map is well-defined, because 
it does not depend on the choice of a representative in the class $[x]_{\mu'}$. Indeed, let 
$[x]_{\mu'}=[x']_{\mu'}$, that is, $\mu'(x,x')=0$. Then $\mu(x,x')\leqslant \mu'(x,x')=0$, whence 
$\mu(x,x')=0$. This means that $[x]_{\mu}=[x']_{\mu}$, that is, 
$p_{\mu\mu'}([x]_{\mu'})=p_{\mu\mu'}([x']_{\mu'})$. It is easy to verify that the map 
$p_{\mu\mu'}\colon X_{\mu'}\to X_\mu$ satisfies the conditions $p_{\mu\mu'}p_{\mu'}=p_{\mu}$ and 
$p_{\mu\mu'}p_{\mu'\mu''}=p_{\mu\mu''}$ for any $\mu, \mu', \mu'' \in \mathcal{P}$, $\mu\leqslant 
\mu'\leqslant \mu''$.

Thus, $\{X_\mu, p_{\mu\mu'}\}$ is an inverse system of invariant metrizable $G$-spaces.
The $G$-space $X_\mu$ is  isometrically and equivariantly
embedded in some normed $G$-space $M(X_\mu)$ as a closed subspace, and the equivariant map
$p_{\mu\mu'}\colon X{_\mu'}\to X{_\mu}$ extends to equivariant maps 
$\bar{p}_{\mu\mu'}\colon M(X{_\mu'})\to M(X{_\mu})$, which satisfy the relation 
$\bar{p}_{\mu\mu'}\bar{p}_{\mu'\mu''}=\bar{p}_{\mu\mu''}$ (see Gevorgyan's paper 
\cite{gev-anequiv}). Thus, we obtain the inverse system $\{M(X_\mu), \bar{p}_{\mu\mu'}\}$ of normed 
$G$-spaces.

Let $A$ be the set of all pairs $(\mu, U)$, where $\mu\in \mathcal{P}$, and let $U$ be an open 
invariant neighborhood of the metrizable $G$-space $X_{\mu}$ in the normed $G$-space $M(X_\mu)$.  
On the set $A$ we define a natural order by setting $\alpha\leqslant \alpha'$ for $\alpha=(\mu, 
U)$ and $\alpha'=(\mu', U')$ if and only if $\mu\leqslant \mu'$ and 
$\bar{p}_{\mu\mu'}(U')\subset U$. Now we set $X_\alpha=U$ and define an equivariant map 
$p_{\alpha\alpha'}\colon X_{\alpha'}\to X_{\alpha}$ by 
$p_{\alpha\alpha'}=\bar{p}_{\mu\mu'}|_{U'}\colon U'\to U$. Note that if $\mu=\mu'$, then 
$\bar{p}_{\mu\mu'}=id\colon M(X{_\mu})\to M(X{_\mu})$; therefore, $p_{\alpha\alpha'}=i\colon 
U'\hookrightarrow U$.

Thus, we have constructed the inverse system $\{X_\alpha, \mathbf{p}_{\alpha\alpha'}, A\}$ 
in the category H-G-ANR. Let us prove that it is equivariantly associated with the $G$-space~$X$.

For each $\alpha=(\mu, U)\in A$, we define an equivariant map $p_{\alpha}\colon X\to X_{\alpha}=U$
by $p_{\alpha}=ip_{\mu}$, where $p_\mu\colon X\to X_\mu$ is the quotient map and $i\colon 
X_\mu \to U$ is an embedding. It is easy to see that $p_\alpha=p_{\alpha\alpha'}p_{\alpha'}$ 
if $\alpha \leqslant \alpha'$.

Now let $f\colon X\to Q$ be a continuous equivariant map, where $Q$ is any $G$-ANR. 
Consider an invariant metric $\rho$ on the space $Q$, whose existence follows from 
the compactness of the group $G$. On $X$ we define a continuous invariant pseudometric $\mu$ 
by $\mu(x,x')=\rho(f(x),f(x'))$. Note that the map $\varphi\colon X_\mu \to Q$ given by 
$\varphi([x]_\mu)=f(x)$ is an equimorphism of the $G$-spaces $X_\mu$ and $f(X)$, and 
$\varphi p_\mu=f$. Since $X_\mu$ is a closed invariant subset of the normed $G$-space $M(X_\mu)$ 
and $Q$ is a $G$-ANR, it follows that there exists an equivariant extension $h\colon  U\to Q$ of 
the map $\varphi\colon X_\mu \to Q$, where $U$ is an invariant neighborhood of $X_\mu$ in 
$M(X_\mu)$. In other words, $h|_{X_\mu}=\varphi$, or $h i=\varphi$.

Take any index $\alpha=(\mu, U)\in A$. We have $X_\alpha=U$, $p_{\alpha}=ip_{\mu}$, and 
$h p_\alpha=f$. Indeed, $h p_\alpha = h ip_{\mu}= \varphi p_{\mu}=f$. Now suppose 
that $h_0, h_1\colon X_\alpha \to Q$ is an equivariant map to the G-ANR $Q$ such that $h_0 
p_\alpha \simeq_G h_1 p_\alpha$. Since $X_\alpha =U$ and $p_{\alpha}=ip_{\mu}$, it follows that  
the equivariant maps $h_0|_{X_\mu}$ and $h_1|_{X_\mu}$ are equivariantly homotopic. Hence 
there exists an invariant neighborhood $V\subset U$ of the closed invariant subset $X_\mu$ such 
that $h_0|_V$ and $h_1|_V$ are equivariantly homotopic as well. This means that 
$h_0p_{\alpha\alpha'}\simeq_G h_1p_{\alpha\alpha'}$, where $\alpha'=(\mu, V)$. 
\end{proof}

\section{The Shape Classification of Spaces}

The shape category, as any other category,  generates a classification of all of its 
objects, that is, a shape classification of topological spaces. Of primary 
interest are classes of spaces in which the shape classification coincides with the homotopy 
classification. One of such classes is the class of ANRs. However, there are also other 
classes for which the shape classification coincides even with the topological one. An example  
is the class of all zero-dimensional spaces (see Godlewski's paper \cite{godlewski-solenoids} 
and Marde\v{s}i\'{c} and Segal's paper~\cite{mard-seg-3}).

\begin{theorem}
Zero-dimensional spaces $X$ and $Y$ have the same shape if and only if they are homeomorphic.
\end{theorem}

\begin{corollary}\label{cor-aleph}
There exist $\aleph_1$ countable compact metrizable spaces with pairwise different shapes.
\end{corollary}

Indeed, all countable compact metrizable sets are zero-dimensional, and  
Mazurkiewicz and Sierpi\'{n}ski \cite{mazur-sierp} proved that there are $\aleph_1$ different 
topological types of countable compact metrizable spaces.

Corollary~\ref{cor-aleph} leads to the conclusion that the number of different shapes of compact 
sets in $\mathbf{R}^1$ is uncountable.

Yet another class of spaces in which the shape classification coincides with 
the topological one is the class of all $P$-adic solenoids $S_P$, where $P=(p_1, p_2, \ldots)$ 
is a sequence of primes. Recall that a $P$-adic solenoid is defined as the limit of an inverse 
sequence $\{X_n, p_{n n+1}, N\}$, where $X_n=S^1$ and $p_{n n+1}$ is a map of degree $p_n$ for each 
$n\in \mathbb{N}$.

\begin{theorem}[\cite{godlewski-solenoids}, \cite{mard-seg-3}]
Solenoids $S_P$ and $S_Q$ have the same shape if and only if they are homeomorphic.
\end{theorem}

Godlewski \cite{godlewski-solenoids} showed that the shape of any solenoid is completely determined 
by its first cohomology group, in much the same way that the shape of any plane continuum 
is determined by its first Betti number.

Another interesting case in which shape theory gives nothing new is that of all 
compact connected Abelian groups. The shape morphisms between groups are in one-to-one  
correspondence with the continuous homomorphisms of these groups. Moreover, Keesling 
\cite{keesling-10} proved the following theorem.

\begin{theorem}
Let $X$ and $Y$ be compact connected Abelian groups. Then $\sh X = \sh Y$ if and only if $X$ 
and $Y$ are isomorphic (algebraically and topologically). 
\end{theorem}

It should be mentioned that this result turns out to be very useful for 
constructing various counterexamples in shape theory. The class of compact connected Abelian 
groups is substantially larger than the class of all solenoids. Moreover,  the following 
theorem of Keesling is valid \cite{keesling-7} (see also~\cite{eber-gor-mack}).

\begin{theorem}\label{th-kess}
Let $X$ be a $T^n$-like continuum, where $T^n = S^1 \times \ldots \times S^1$ is the 
$n$-torus. Then $X$ has the shape of a compact connected Abelian topological group. 
\end{theorem}

However, if $\Pi$ is the family of all compact connected Lie groups and $X$ is a $\Pi$-like 
continuum, then the shape of $X$ may be different from that of a compact connected topological  
group. The corresponding example was constructed in~\cite{keesling-7}.

In the case of plane continua, shape depends only on the first Betti number of a continuum, i.e., 
on the number of domains into which the plane is separated by this continuum.

\begin{theorem}\cite{borsuk-rus}\label{th-borsuk}
Two plane continua $X$ and $Y$ have the same shape if and only if their first Betti numbers 
coincide. 
\end{theorem}

In particular, $\sh X \geqslant \sh Y$ if and only if the number of domains in 
$\mathbf{R}^2\backslash X$ is greater than or equal to that of domains in $\mathbf{R}^2\backslash 
Y$. Therefore, for any two plane continua $X$ and $Y$, we have either $\sh X = \sh Y$, $\sh 
X < \sh Y$, or $\sh X > Y$. However, for continua in Euclidean 3-space 
$\mathbf{R}^3$, the situation is different: there exist continua $X, 
Y \subset \mathbf{R}^3$ such that $\sh X \leqslant \sh Y$ and $\sh X \geqslant \sh Y$ but $\sh X 
\neq \sh Y$.

Theorem~\ref{th-borsuk} implies the existence of countably many different shapes 
of plane continua. Representatives of these shapes are a point (the trivial shape), the wedge of 
$n$ circles for every $n \in \mathbb{N}$, and the wedge of infinitely many circles. This gives a 
complete shape classification of plane continua.

It is easy to show that the family of all different shapes of compact sets in the plane 
has cardinality $2^{\aleph_0}$. Godlewski \cite{godlewski-on} proved that in $\mathbf{R}^3$ there 
exist $2^{\aleph_0}$ continua with pairwise different shapes (see also~\cite{mard-seg-3}). In fact, 
these different shapes can be found among solenoids.

Complete shape classifications of $\Pi$-similar continua for various classes $\Pi$ can be 
found in \cite{mard-seg-3}, \cite{segal-on}, \cite{segal-shapeclass}, \cite{segal-shape}, 
\cite{handle-segal}, \cite{handle-segal-finite}, \cite{handle-segal-on}, \cite{keesling-7}, 
\cite{eber-gor-mack}, and \cite{watanabe-shape}.

\section{Complement Theorems in Shape Theory}

Shape theory is a very useful tool for solving classical problems of infinite-dimensional and 
geometric topology. This has become quite clear after deep results of Chapman \cite{chapman-on}, 
\cite{chapman-4},  Geoghegan and Summerhill \cite{geoghegan-summ-1},  Edwards 
\cite{edwardsR.-char}, West \cite{west-map}, and other authors, which have demonstrated 
the effectiveness of  shape theory methods in geometric topology. On the other 
hand, methods of infinite-dimensional topology are very useful for studying the shapes 
of metrizable compact sets. As far as we know, in homotopy theory, these methods were first 
applied by Borsuk \cite{borsuk-homotopy}. He proved a theorem about 
the homotopy type of a quotient space of Euclidean space $\mathbf{R}^n$ by using Klee's theorem 
\cite{klee} on extending homeomorphisms of compact sets in a Hilbert space to the entire 
Hilbert space. These methods were also applied by Henderson~\cite{henderson}.

A special role in the application of methods of infinite-dimensional topology to 
shape theory is played by the so-called complement theorems, which answer 
the following question in various situations: When are the complements  
$M\backslash X$ and $M\backslash Y$ of compact sets $X$ and $Y$ embedded in 
an ambient space $M$ in a special way homeomorphic? One of the first results is this 
kind was Borsuk's theorem \cite{borsuk-rus} that if $X$ and $Y$ are plane continua, 
then the complements $\mathbf{R}^2\backslash X$ and $\mathbf{R}^2\backslash Y$ are homeomorphic if 
and only if $\sh(X)=\sh(Y)$. Remarkable results in this direction were also obtained by Chapman 
\cite{chapman-on}, \cite{chapman-4}. In essence, Chapman's theorems assert that, under certain 
constraint on the embedding of two compact spaces $X$ and $Y$ in the Hilbert cube $Q$, these spaces 
have the same shape if and only if their complements are homeomorphic. Therefore, if $X$ 
and $Y$ are absolute neighborhood retracts, then they have the same homotopy type if and 
only if their complements are homeomorphic.

We represent the Hilbert cube $Q$ in the form the product $\prod\limits_{n=1}^\infty I_n$, 
where $I_n=[0,1]$. The set $s=\prod\limits_{n=1}^\infty \mathring{I}_n$, where 
$\mathring{I}_n=(0,1)$, is called the \emph{pseudointerior} of the Hilbert cube~$Q$.

Recall that a compact subset $X$ of the Hilbert cube $Q$ is called a \emph{$Z$-set} if the identity 
map $1_Q$ can be approximated arbitrarily closely by continuous maps from $Q$ to $Q\backslash X$, 
i.e., for any $\varepsilon >0$, there exists a continuous map $f\colon  Q \to Q\backslash X$ such 
that $d(f, 1_Q)< \varepsilon$. This notion was introduced by Anderson \cite{anderson} and is one 
of the fundamental notions of the theory of $Q$-manifolds~\cite{chapman-1}.

Note that any compact metrizable space is embedded in the Hilbert cube as a $Z$-set.
Compact sets lying in the pseudointerior $s$ of the Hilbert cube $Q$ form an important class 
of $Z$-sets. The following theorem shows that any $Z$-set in the Hilbert cube $Q$ can be mapped 
to the pseudointerior $s$ by a homeomorphism $h\colon Q \to Q$.

\begin{theorem}
Let $X\in Q$ be a $Z$-set. Then, for any $\varepsilon >0$, there exists a homeomorphism 
$h\colon Q \to Q$ such that $h(X)\subset s$ and $d(h, 1_Q)< \varepsilon $. 
\end{theorem}

This theorem, which was proved by Anderson~\cite{anderson}, is exceptionally important, 
because it reduces studying $Z$-sets to the theory of compact sets in the pseudointerior~$s$.

Chapman \cite{chapman-on} proved that the shape of any $Z$-set (in particular, of any compact set 
in $s$) depends only on the topological type of its complement. To be more precise, the following 
remarkable \emph{complement theorem} is valid (see~\cite{chapman-on}).

\begin{theorem}\label{th-chapman}
Let $X$ and $Y$ be any $Z$-sets in the Hilbert cube $Q$. Then they have the same shape if and only 
if their complements $Q\backslash X$ and $Q\backslash Y$ are homeomorphic. 
\end{theorem}

If  $X$ and $Y$ are absolute neighborhood retracts, then they  have the same
homotopy type if and only if their complements are homeomorphic.

To prove his theorem, Chapman applied deep results on $Q$-manifolds (see 
\cite{anderson}, \cite{anderson-schory}, \cite{anderson-hend-west}, and \cite{wong}).

Note that a similar assertion for subsets of the Hilbert space $l_2$ is false, 
because, according to one of Anderson's theorems~\cite{anderson-strongly}, $l_2$ is homeomorphic to 
$l_2\backslash X$ for any compact set $X\subset l_2$.

Chapman \cite{chapman-4} also proved the first finite-dimensional analogue of  
Theorem~\ref{th-chapman} about complements. This theorem has attracted attention of many experts 
in geometric topology, who proved new finite-dimensional complement theorems. We mention papers 
 \cite{geoghegan-summ-1} by Geoghegan and Summerhill,   \cite{venema} by Venema, 
 \cite{iv-sh-ven} by Ivan\v{s}i\'{c}, Sher, and Venema, and  \cite{mroz} by Mrozik. In all these 
papers, it was assumed that compact metrizable spaces $X$ and $Y$ are appropriately 
embedded in $\mathbf{R}^n$ and satisfy certain dimension conditions, and it was proved that 
$X$ and $Y$ have the same shape if and only if their complements $\mathbf{R}^n\backslash X$ and 
$\mathbf{R}^n\backslash Y$ are homeomorphic.

\begin{theorem}[Chapman \cite{chapman-4}]
Let $X$ and $Y$ be compact metrizable spaces of dimension $\leqslant k$. Then the following 
assertions hold:

(a) for any $n\geqslant 2k+2$,  there exist embeddings 
$i\colon  X \to \mathbf{R}^n$ and $j \colon Y \to \mathbf{R}^n$ such that  
if $\sh (X) = \sh (Y)$, then $\mathbf{R}^n \backslash i(X) \cong \mathbf{R}^n \backslash j(Y)$;

(b) for any $n \geqslant 3k + 3$, there exist embeddings $i\colon  X \to \mathbf{R}^n$ and $j 
\colon  Y \to \mathbf{R}^n$ such that if $\mathbf{R}^n \backslash i(X) \cong 
\mathbf{R}^n \backslash j(Y)$, then $\sh (X) = \sh (Y)$. 
\end{theorem}

To state and prove their complement theorem, Geoghegan and Summerhill \cite{geoghegan-summ-1} 
introduced the notions of an $\varepsilon$-push and a strong $Z_k$-set. Let $X$ 
be a closed subset of $\mathbf{R}^n$. A homeomorphism $h\colon \mathbf{R}^n \to \mathbf{R}^n$ 
is called an \emph{$\varepsilon$-push} of the pair $(\mathbf{R}^n, X)$ if 
there exists an isotopy $H\colon \mathbf{R}^n\times I \to \mathbf{R}^n$ such that $d(x,H_t(x)< 
\varepsilon)$ for all $x\in \mathbf{R}^n$ and $t\in I$ (this isotopy is called 
an $\varepsilon$-isotopy), $H_0=1$, $H_1=h$, and $H_t(x)=x$ for all $t\in I$ and all $x$ 
satisfying condition $\dist (x, X) \geqslant \varepsilon$.

A closed subset $X$ of $\mathbf{R}^n$ is called a \emph{strong $Z_k$-set} $(k\geqslant 0)$ if, for 
each compact polyhedron $P$ in $\mathbf{R}^n$ of dimension $\dim P \leqslant k+1$ and any 
$\varepsilon > 0$, there exists an $\varepsilon$-push $h$ of the pair 
$(\mathbf{R}^n, X)$ such that $P \cap h(X) = \emptyset $.

\begin{theorem}[Geoghegan and Summerhill \cite{geoghegan-summ-1}]
Let $X$ and $Y$ be compact metrizable spaces 
being strong $Z_{n-k-2}$-sets in $\mathbf{R}^n$ $\left(k \geqslant 0, n \geqslant 
2k+2\right)$.

Then the following conditions are equivalent:

(i) $\sh (X) = \sh (Y)$;

(ii) the pairs $(\mathbf{R}^n/X, \{X\})$ and $(\mathbf{R}^n/Y, \{Y\})$ are homeomorphic;

(iii) the triples $(\mathbf{R}^n/X,(\mathbf{R}^n/X)\backslash \{X\}, \{X\})$ and 
$(\mathbf{R}^n/Y,(\mathbf{R}^n/Y)\backslash \{Y\}, \{Y\})$ have the same homotopy type;

(iv) $\mathbf{R}^n\backslash X \cong \mathbf{R}^n\backslash Y$.
\end{theorem}

In relation to this theorem, it is important that \emph{any metrizable compact set $X$ 
of dimension $\dim X \leqslant k$ is embedded in $\mathbf{R}^n$ $(n \geqslant 2k + 1)$ as a strong 
$Z_{n-k-2}$-set}.

Note that the property of being a strong $Z_k$-set is not homotopy invariant. However, 
homotopy-invariant properties often imply the property of being a strong $Z_{n-k-2}$-set. 
This has enabled Geoghegan and Summerhill \cite{geoghegan-summ-1} to prove the 
following complement theorem, which is more convenient for applications.

\begin{theorem}\label{th_geg-sam}
Let $X$ and $Y$ be nonempty compact sets in $\mathbf{R}^n$ such that their complements 
$\mathbf{R}^n\backslash X$ and $\mathbf{R}^n\backslash Y$  are uniformly locally 1-connected and 
$\max \{\dim X, \dim Y\}\leqslant k$,  $\max \{2k+2, 5\}\leqslant n$. Then $\sh (X) = \sh (Y)$  if 
and only if $\mathbf{R}^n\backslash X \cong \mathbf{R}^n\backslash Y$. 
\end{theorem}

Recall that a metric space $X$ is said to be \emph{uniformly locally $k$-connected} 
(written as $X \in ULC^k$) if, for any $\varepsilon > 0$, there exists a $\delta >0$ such that 
any map $f\colon S^k \to X$ satisfying the condition $\diam f(S^k) < \delta$ is homotopic to a  
constant map on any set $V\subset X$ of diameter $\diam V < \varepsilon$.

\begin{corollary}
Let $X$ and $Y$ be closed $k$-dimensional submanifolds of $\mathbf{R}^n$, where $\max \{2k+2, 
5\}\leqslant n$. Then $X$ is homotopy equivalent to $Y$ if and only if $\mathbf{R}^n\backslash X 
\cong \mathbf{R}^n\backslash Y$. 
\end{corollary}

The condition $ULC^1$ in Theorem~\ref{th_geg-sam} cannot be omitted, because the complement 
of Blankinship's wild arc \cite{blank} in $\mathbf{R}^n$, $n \geqslant 3$, is not simply 
connected. It is also important that the codimension be sufficiently large. Henderson 
\cite{henderson} constructed an example of non--homotopy equivalent two-dimensional compact 
polyhedra with homeomorphic complements in $\mathbf{R}^3$.

In \cite{cor-dav-duv} the  so-called \emph{small loops condition} (SLC) was introduced. A compact 
space $X\subset \mathbf{R}^n$ satisfies the SLC if, for any neighborhood $U$ of 
$X$ there exists a smaller neighborhood $V$ ($X\subset V \subset U$) and a number $\varepsilon 
> 0$ such that any loop in $V\backslash X$ of diameter $< \varepsilon $ is homotopic to zero in 
$U\backslash X$. This notion is a generalization of McMillan's \emph{cellularity criterion} 
(CS) \cite{mcmillan}, which characterizes cellular embeddings in manifolds. The CC  
is obtained from the SLC at $\varepsilon = \infty$.

Using results of~\cite{cor-dav-duv}, Hollingsworth and Rushing 
\cite{holl-rush} were able to prove the following generalization of Theorem~\ref{th_geg-sam}.

\begin{theorem}\label{th_hol-rush}
Let $X,Y \subset \mathbf{R}^n$ be  compact sets satisfying the SLC, and let $\max 
\{\dim X, \allowbreak \dim Y\}\leqslant k$, where  $\max \{2k+2, 5\}\leqslant n$. Then $\sh (X) = \sh (Y)$  if 
and only if $\mathbf{R}^n\backslash X \cong \mathbf{R}^n\backslash Y$. 
\end{theorem}

This theorem has two interesting corollaries.

\begin{corollary}
Let $X,Y \subset \mathbf{R}^n$ be compact ANRs satisfying the SLC, and let $\max 
\{\dim X,\allowbreak \dim Y\}\leqslant k$, where $\max \{2k+2, 5\}\leqslant n$. Then $X$ and $Y$ have the same 
homotopy type if and only if $\mathbf{R}^n\backslash X \cong \mathbf{R}^n\backslash Y$. 
\end{corollary}

\begin{corollary}
Suppose that compact spaces $X,Y \subset \mathbf{R}^n$ are homeomorphic and satisfy the 
SLC. If $\dim X = \dim Y \leqslant k$ and  $\max \{2k+2, 5\}\leqslant n$, 
then $\mathbf{R}^n\backslash 
X \cong \mathbf{R}^n\backslash Y$. 
\end{corollary}

Venema \cite{venema} has succeeded in replacing the  dimension $\dim$  in Theorem~\ref{th_hol-rush}
by the shape dimension $\sd$. For this purpose, he defined the \emph{inessential loop condition} 
(ILC). A compact subset $X$ of a manifold $M$ satisfies the ILC if, for any neighborhood $U$ 
of $X$ in $M$, there exists a smaller neighborhood $V$, $X\subset V\subset U$, such that any loop 
in $V\backslash X$ which is inessential (homotopic to  zero) in $V$ is also inessential in 
$U\backslash X$.  Note that the ILC implies the SLC. If $X\subset M^n$ and $\dim X \leqslant n - 
2$, then these two conditions are equivalent.

\begin{theorem}[Venema \cite{venema}]\label{th-venema}
Suppose that compact sets $X,Y \subset \mathbf{R}^n$ satisfy the ILC, $\max \{\sd X, \sd 
Y\}\leqslant k$, and $\max \{2k+2, 5\}\leqslant n$. Then $\sh (X)=\sh (Y)$ if and only if 
$\mathbf{R}^n\backslash X \cong \mathbf{R}^n \backslash Y$. 
\end{theorem}

In the case where $\dim X \leqslant n-3$ and $Y$ is a finite polyhedron of dimension 
$\leqslant k$ such that $2k+2\leqslant n$, this theorem was proved in~\cite{cor-dav-duv}.

The most general complement theorem was proved by Ivan\v{s}i\'{c}, Sher, and Venema~\cite{iv-sh-ven}.

\begin{theorem}[\cite{iv-sh-ven}]
Suppose that compact sets $X,Y \subset \mathbf{R}^n$ are shape $r$-connected and satisfy the ILC. 
Suppose also that $\max \{\sd X, \sd Y\}\leqslant k$ and $n \geqslant \max \{2k+2-r, 5\}$. 
If $n \geqslant k+3$, then $\sh (X)=\sh (Y)$ implies $\mathbf{R}^n\backslash X \cong 
\mathbf{R}^n \backslash Y$. The converse is true if $n \geqslant k+4$. 
\end{theorem}

This theorem generalizes almost all preceding results.

Complement theorems in various categories and in more general situations were proved by 
Mrozik~\cite{mroz}.

\section{Movability and Other Shape Invariants}

The notion of movability for compact metrizable spaces was introduced by Borsuk~\cite{borsuk-4}.

\begin{definition}\label{def-mov1}
A compact metrizable space $X$ embedded in an AR $M$ is said to be \emph{movable} if, given 
any neighborhood $U$ of $X$, there exists a neighborhood $U'\subset U$ of $X$ 
such that, for any neighborhood $U''\subset U$ of $X$, there exists 
a homotopy $H\colon U'\times I \to U$ satisfying the conditions $H(x,0)=x$ and $H(x,1)\in U''$ 
for all $x\in U'$. 
\end{definition}

This definition does not depend on the choice of the AR $M$ and the embedding of $X$ in~$M$.

For arbitrary topological spaces, the notion of movability was defined by Marde\v{s}i\'{c} and 
Segal \cite{mard-seg-2} in terms of ANR-systems.

\begin{definition}\label{def-mov2}
An inverse ANR-system  $\mathbf{X}=\{X_\alpha , \mathbf{p}_{\alpha \alpha '}, A\}$ is said to be 
\emph{movable} if the following condition holds:

(M) given any $\alpha \in A$, there is an $\alpha ' \in A$, $\alpha ' \geqslant \alpha $, 
such that, for any $\alpha'' \in A$, $\alpha'' \geqslant \alpha $, there exists a homotopy 
class $\mathbf{r}^{\alpha ' \alpha ''}\colon X_{\alpha '}\to X_{\alpha ''}$ satisfying the 
condition $\mathbf{p}_{\alpha \alpha'}=\mathbf{p}_{\alpha \alpha''}\circ 
\mathbf{r}^{\alpha ' \alpha ''}$.

An inverse ANR-system $\mathbf{X} = \{X_\alpha , \mathbf{p}_{\alpha \alpha '}, A\}$ is  
\emph{uniformly movable} if

(UM) given any $\alpha \in A$, there exists an $\alpha ' \in A$, $\alpha ' \geqslant \alpha $, and 
a morphism $\mathbf{r}\colon  X_{\alpha '}\to \mathbf{X}$ in the category pro-H-CW such that 
$\mathbf{p}_{\alpha}\circ \mathbf{r} = \mathbf{p}_{\alpha \alpha'}$, where $\mathbf{p}_{\alpha} 
\colon  \mathbf{X} \to X_\alpha $ is the morphism in  pro-H-CW generated by the identity 
homotopy class $\mathbf{1}_{X_\alpha}$.

A topological space $X$ is said to be \emph{(uniformly) movable} if there exists an associated 
(uniformly) movable ANR-system  $\{X_\alpha , \mathbf{p}_{\alpha \alpha '}, A\}$. 
\end{definition}

If a space $X$ is (uniformly) movable, then all  ANR-systems associated with it are (uniformly) 
movable, i.e., the notion of (uniform) movability depends only on the space $X$ itself.  In the 
case of compact metrizable spaces, Definitions~\ref{def-mov1} and~\ref{def-mov2} are equivalent.

The following theorem shows that the notion of a movable space can be defined without resorting 
to associated inverse systems.

\begin{theorem}[Gevorgyan \cite{gev-obodnom}]\label{def-mov3}
A topological space $X$ is movable if and only if the following condition holds:

(*) given any homotopy  class $\mathbf{f}\colon X\to Q$, where $Q\in \mathrm{ANR}$, there are 
homotopy classes $\mathbf{f'}\colon X\to Q'$ and $\boldsymbol{\eta} \colon Q'\to Q$, where $Q'\in 
\mathrm{ANR}$ and $\mathbf{f}=\boldsymbol{\eta} \mathbf{f'}$, such that, for any homotopy classes 
$\mathbf{f''}\colon X\to Q''$ and $\boldsymbol{\eta}' \colon Q''\to Q$, where $Q''\in 
\mathrm{ANR}$ and $\mathbf{f}=\boldsymbol{\eta}' \mathbf{f''}$, there exists a homotopy class 
$\boldsymbol{\eta}''\colon Q'\to Q''$ satisfying the condition $\boldsymbol{\eta} =\boldsymbol{\eta}' 
\boldsymbol{\eta}''$ (see Diagram~1). 
\\ % 
$$ 
\xymatrix { & & Q'   
\ar@/_1pc/[dll]_{\boldsymbol{\eta}} \ar@{-->}[dd]^{\boldsymbol{\eta}''} \\ Q  & 
\ar[l]_{\hspace{4mm}\mathbf{f}} X \ar[ur]_{\hspace{-2mm}\mathbf{f'}} 
\ar[dr]^{\hspace{-2mm}\mathbf{f''}} \\ & & Q'' \ar@/^1pc/[ull]^{\boldsymbol{\eta}'}& } 
$$ % 
\begin{center} \text{Diagram 1} \end{center}

\end{theorem}

Movability is a monotone shape invariant, i.e., \emph{if $\sh X \leqslant \sh Y$ and $Y$ is 
movable, then so is $X$}. Movability is preserved under taking products (Borsuk 
\cite{borsuk-4} and Bogatyi \cite{bogat-ob}) and suspensions (Borsuk 
\cite{borsuk-7}), but it is not preserved under intersection, union (Cox \cite{cox}), 
and cell maps (Keesling \cite{keesling-3}). If each component of a compact space $X$ is movable, 
then the whole space $X$ is movable. However, the converse is not true: there exists a 
movable compact space which has a non-movable component (see~\cite{borsuk-rus}).

Movable spaces are rather many. All ANRs are movable. Moreover, all \emph{stable spaces} (i.e., 
spaces having the shape of an ANR) are movable. All compact sets in the plane are movable as well 
\cite{borsuk-4}. However, in $\mathbf{R}^3$ there exist non-movable continua; e.g., so are all 
solenoids \cite{borsuk-rus}. This follows from the non-movability of the first homology pro-group 
of any solenoid. The point is that movability is preserved under functorial passages. 
Therefore, if a space $X$ is movable, then so are its homotopy pro-groups $\pro-\pi_n(X,*)$ and 
homology pro-groups pro-$H_n(X)$.

Movability, which can be defined in any pro-category (see \cite{mosz-uniformly}), is important 
because the passage to the limit in movable inverse systems can be performed without losing 
the algebraic information about the system. In the case of a shape morphism $F\colon  (X, 
*) \to (Y, *)$ of movable compact metrizable spaces, the induced homomorphisms 
$\check{\pi}_n(F) \colon \check{\pi}_n(X,*) \to \check{\pi}_n(Y,*)$ and 
$\pro-\pi_n(F)\colon \pro-\pi_n(X,*) \to \pro-\pi_n(Y,*)$ of the shape groups and pro-groups are, 
or are not, isomorphisms simultaneously (see Theorem~\ref{th-mov-iso}). Therefore, in the case 
of movable spaces, some theorems remain  valid under the replacement of homotopy 
pro-groups by shape groups. This is the case with the shape versions of Whitehead's theorem 
(see Theorem~\ref{th-witehead-mov}), Hurewicz' theorem (see Theorem~\ref{th-hur-mov}), and 
so on.

In the category pro-Set movability has a simple characterization: an inverse system $\mathbf{X} 
=\{X_\alpha, p_{\alpha\alpha'}, A\}\in \pro-\Set$  is movable if and only if the following 
Mittag-Leffler (ML) condition holds: \emph{for any $\alpha\in A$, there exists an 
$\alpha'\geqslant \alpha$ such that $p_{\alpha\alpha'} (X_{\alpha'}) = 
p_{\alpha\alpha''}(X_{\alpha''})$ for any $\alpha''\geqslant \alpha'$}. All movable pro-groups  
satisfy the Mittag-Leffler condition. However, the converse is not true (see~\cite{mard-seg-4}).

The family of all movable compact spaces is \emph{shape bounded}; to be more precise, 
the following theorem of Spie\.z is valid~\cite{spiez-majorant}.

\begin{theorem}\label{th-spiez}
There exists a movable compact space $X_0$ such that $\sh (X) \leqslant \sh (X_0)$ for any movable 
compact space $X$. 
\end{theorem}

Such a compact space $X_0$ is called a \emph{majorant} of the family of all movable compact spaces.
This theorem implies that the family of all compact sets in the plane is shape bounded, or has 
a majorant. This is not true for the family of all compact sets in $\mathbf{R}^3$. 
This follows from a result of Borsuk and Holszty\'{n}ski \cite{borsuk-holsztynski}, according 
to which there does not exist a compact space whose shape is larger than that of any solenoid.

In equivariant shape theory, the counterpart of Theorem~\ref{th-spiez} is false, i.e., 
\emph{the class of all $G$-movable compact spaces has no majorants}. However,  the following 
theorem is valid.

\begin{theorem}[Gevorgyan \cite{gev-majorant}]\label{th-gev1}
Let $G$ be a second-countable compact group. Then, in any class of weakly shape 
comparable $G$-movable compact spaces, there exists a majorant. 
\end{theorem}

The \emph{weak shape comparability} of $G$-spaces $X$ and $Y$ means that there exist $G$-shape 
morphisms both from $X$ in $Y$ and from $Y$ to $X$. The proof of the last theorem is based on 
McCord's construction \cite{mccord} of a universal compact space and on the equivariant 
counterpart of Brown's theorem~\cite{gev-majorant}.

The first results on equivariant movability were obtained by the author in  
\cite{gev-majorant}--\cite{gev-freudenthal} and~\cite{gev-some}, where, in particular, 
the following theorems were proved.

\begin{theorem}
Let $G$ be a compact Lie group, and let $X$ be a $G$-movable metrizable $G$-space. Then 
$X$ is $H$-movable for any closed subgroup $H$ of~$G$. 
\end{theorem}

The $G$-movability of a metrizable $G$-space implies also the movability of the 
$H$-fixed point set for any closed subgroup $H$ of $G$. In particular, \emph{the 
$G$-movability of a metrizable $G$-space $X$ implies its movability}. The converse is not 
true even if the acting group $G$ is the cyclic group $\mathbb{Z}_2$ 
(see~\cite[Example 5.1]{gev-some}).

\begin{theorem}
Let $G$ be a compact group, and let $X$ be a metrizable $G$-space. If $X$ is $G$-movable, then, for 
any closed normal subgroup $H$ of $G$, the $H$-orbit space $X|_H$ is $G$-movable as well. 
\end{theorem}

In particular, the $G$-movability of a $G$-space $X$ implies the movability of the orbit space 
$X|_G$. The converse is not true. However, if a compact Lie group acts freely on a metrizable 
space $X$, then $G$-movability is equivalent to movability. The assumptions of the theorem that 
the acting group $G$ is Lie and the action is free are essential. 

The movability of topological groups was studied by Keesling \cite{keesling-6}, \cite{keesling-10}, 
\cite{keesling-the} and by Kozlowski and Segal \cite{kozlowski-segal-2}. In particular, Keesling 
\cite{keesling-6} proved the following theorem.

\begin{theorem}
A connected compact Abelian group is movable if and only if it is locally connected.
\end{theorem}

For non-Abelian groups, this theorem does not hold (see~\cite{keesling-6}).

The equivariant movability of topological groups was studied by the author in~\cite{gev-equiv.mov}. 
We mention the following result.

\begin{theorem}[Gevorgyan \cite{gev-equiv.mov}]
A second-countable compact group $G$ is Lie if and only if it is $G$-movable.
\end{theorem}

This theorem gives, in particular, new examples of movable but equivariantly non-movable  
$G$-spaces (see~\cite{gev-equiv.mov}).

\emph{Movable shape morphisms} were defined and studied by Gevorgyan and Pop \cite{Gev-Pop-1}. 
A morphism $(\mathbf{f}, \varphi) \colon  \{X_\alpha, \mathbf{p}_{\alpha\alpha'}, A\} \to 
\{Y_\beta, \mathbf{q}_{\beta\beta'}, B\}$ in the category pro-H-CW is said to be \emph{movable}  
if, given any $\beta \in B$, there is an $\alpha \in A$, $\alpha \geqslant \varphi(\beta)$, 
such that, for any $\beta' \geqslant \beta$, there exists a homotopy class $\mathbf{u}\colon 
X_{\alpha} \to Y_{\beta'}$ satisfying the condition $\mathbf{f}_\beta  
\mathbf{p}_{\varphi(\beta)\alpha} = \mathbf{q}_{\beta\beta'} \mathbf{u}$. The notion of a movable 
shape morphism agrees with that of a movable space in the 
sense that \emph{a space $X$ movable if and only if so is the identity map $1_X$}. 
Movable morphisms are important, in particular, because 
the movability assumption on a space can sometimes be replaced by the weaker assumption 
that a shape morphism $F\colon X\to Y$ is movable; see, e.g., Whitehead's 
Theorem~\ref{th-gev-pop-1}  for movable morphisms $F\colon X\to Y$.

Movable shape morphisms were also introduced and studied (by different methods and for 
different purposes) by other authors \cite{edwardsA.-Mc}, \cite{cherin}, 
\cite{yagasaki-fiber}, \cite{yagasaki-mov}.

In~\cite{borsuk-rus} Borsuk  introduced the notion of $n$-movability, which is a shape 
invariant, too. A metrizable compact space $X$ lying in AR-space $M$ is said to be 
\emph{$n$-movable} if, given any neighborhood $U$ of $X$ in $M$, there exists a 
neighborhood $U'\subset U$ of $X$ such that, for any neighborhood $U''\subset U$ of $X$, 
any metrizable compact space $K$ of dimension $\dim K \leqslant n$, and any map $f\colon 
K\to U'$, there exists a map $g\colon K \to U''$ homotopic to $f$ in $U$. For arbitrary spaces, this 
notion is introduced by using inverse systems.

\begin{definition}\label{def-n-mov}
An inverse ANR-system  $\{X_\alpha , \mathbf{p}_{\alpha \alpha '}, A\}$ is said to be 
\emph{$n$-movable} if, given any $\alpha \in A$, there is an $\alpha ' \in A$, $\alpha ' 
\geqslant \alpha $, such that, for any $\alpha'' \in A$, $\alpha'' \geqslant \alpha $, and any 
homotopy class $\mathbf{h}\colon P \to X_{\alpha'}$, where $P$ an ANR of dimension $\dim P 
\leqslant n$, there exists a homotopy class $\mathbf{r} \colon  P \to X_{\alpha ''}$ satisfying 
the condition $\mathbf{p}_{\alpha \alpha'} \mathbf{h}=\mathbf{p}_{\alpha \alpha''} 
\mathbf{r}$.

A topological space $X$ is \emph{$n$-movable} if there exists an associated $n$-movable ANR-system 
$\{X_\alpha , \mathbf{p}_{\alpha \alpha '}, A\}$. 
\end{definition}

It is easy to see that $(n+1)$-movability always implies $n$-movability. It is also clear that 
the movability of a space $X$ implies its $n$-movability for any $n$.  If $\dim X \leqslant 
n$, then the converse is also true: the $n$-movability of a space $X$ implies its 
movability (in the case $\sd X \leqslant n$, this was proved in \cite{bogat-ob}). The equivariant 
counterpart of this statement has been proved in the case of a finite acting group $G$: any 
equivariantly $n$-movable $G$-compact set of dimension $\dim X \leqslant n$ is equivariantly 
movable (see Gevorgyan \cite{gev-freudenthal}). Therefore, a finite-dimensional $G$-compact space 
is equivariantly movable if and only if it is equivariantly $n$-movable for all $n$. In the general 
case, the $n$-movability does not imply movability: the Kahn compact space \cite{kahn} is 
not movable, but it is $n$-movable for all $n$ (see \cite{bogat-ob}, 
\cite{kozlowski-segal-3}). Any $LC^{n-1}$ paracompact space is $n$-movable (see 
\cite{kozlowski-segal-1}).

The particularly important case of 1-movability was studied in detail by Dydak 
\cite{dydak-asimple}, McMillan \cite{mcmillan-one}, Krasinkiewicz \cite{krasink-cont}, 
\cite{krasink-minc}, \cite{krasink-local}, and other authors. If a metrizable continuum $(X,*)$ 
is 1-movable for a point $x\in X$, then it is also 1-movable for any other point of $X$.  Moreover, 
if a metrizable continuum $(X,x)$ is 1-movable and $\sh(X)=\sh(Y)$, then $\sh(X,x)=\sh(Y,y)$ for 
any $y\in Y$ (Dydak~\cite{dydak-asimple}). There exists a nontrivial example 
of a metrizable continuum $X$ such that it is 1-movable but $(X,*)$ is not movable
(Dydak \cite{dydak-1-movable}). A continuous image of a 1-movable metrizable continuum is 
1-movable.  The following important criterion for the 1-movability of a metrizable continuum is due 
to Krasinkiewicz~\cite{krasink-local}.

\begin{theorem}
A metrizable continuum $(X,*)$ is 1-movable if and only if its homotopy pro-group 
$\pro-\pi_1(X,*)$ satisfies the Mittag-Leffler condition. 
\end{theorem}

\begin{theorem}
A metrizable continuum $(X,*)$ is 1-movable if and only if it has the shape of a locally connected 
continuum. 
\end{theorem}

It follows from these theorems that a solenoid $X$ does not have the shape of a locally connected 
continuum, because the pro-group $\pro-\pi_1(X)$ does not satisfy the Mittag-Leffler condition.

Many results of shape theory are easy to transfer from the 
case of pointed spaces to that of nonpointed spaces. However, the 
reverse transfer of results is not always easy or even possible (see \cite{geoghegan-the}, 
\cite{mard-seg-4}). For example, the movability of a pointed space $(X,*)$ readily 
implies that of $X$, but the question of whether the converse is true is difficult and has not been 
answered so far. A partial answer is given by the following theorem.

\begin{theorem}[Krasinkiewicz \cite{krasink-cont}]
Let $(X,*)$ be a 1-movable metrizable continuum. Then the movability of $X$ implies that 
of $(X,*)$. 
\end{theorem}

The notion of $n$-movability has initiated \emph{$n$-shape theory}, which was developed by 
Chigogidze in \cite{chigogidze}. This theory is an important tool in the study of $k$-dimensional 
Menger manifolds, which was begun by Bestvina in~\cite{bestvina}.

There are several more modifications of the notions of movability and $n$-movability 
(see papers \cite{bogat-aproc} by Bogatyi,  \cite{kodama-fine} by Kodama, 
\cite{mard-strongly} by Marde\v{s}i\'{c}, and  \cite{mosz-uniformly} by Moszynska; we also 
mention papers \cite{bauer-character} by Bauer,  \cite{dydak-oninternally} by Dydak,  
\cite{kozlowski-segal-2} by Koslowski and Segal,  \cite{spiez-mov} by Spie\.z,  and  
\cite{watanabe-onstrong} by Watanabe). The notion of the \emph{$n$-movability of shape morphisms} 
was introduced and studied by Gevorgyan and Pop~\cite{Gev-Pop-2}.

Yet another important shape invariant is \emph{strong movability}. This notion has played an 
important role, primarily in the study of the stability of topological spaces and 
absolute neighborhood shape retracts. A space $X$ is said to be \emph{strongly movable} if there 
exists an associated \emph{strongly movable ANR-system} $\{X_\alpha , \mathbf{p}_{\alpha 
\alpha '}, A\}$, which means that, given any $\alpha \in A$, there is an $\alpha 
' \geqslant \alpha $ such that, for any $\alpha'' \geqslant \alpha $,  there exists an 
$\alpha^* \geqslant \alpha', \alpha''$ and a homotopy class $\mathbf{r}^{\alpha'\alpha''} \colon  
X_{\alpha'} \to X_{\alpha ''}$ satisfying the conditions $\mathbf{p}_{\alpha \alpha''} 
\mathbf{r}^{\alpha'\alpha''}=\mathbf{p}_{\alpha \alpha'}$ and $\mathbf{r}^{\alpha'\alpha''}   
\mathbf{p}_{\alpha' \alpha^*}=\mathbf{p}_{\alpha'' \alpha^*}$. Obviously, strong 
movability implies movability. Strong movability is preserved under shape domination. For a 
connected space $(X,*)$, strong movability is equivalent to stability (Dydak \cite{dydak-on}, 
Watanabe \cite{watanabe-onstrong}).

The notion of movability was extended to and studied in more general cases by 
Segal \cite{segal-movableshapes}, Shostak \cite{shostak}, Gevorgyan \cite{gev-obodnom}, 
\cite{gev-mov}, Gevorgyan and Pop \cite{Gev-Pop-1}, \cite{Gev-Pop-2}, \cite{Gev-Pop-4}, and Avakyan 
and Gevorgyan~\cite{av-gev}.

Avakyan and Gevorgyan~\cite{av-gev} (see also the author's papers \cite{gev-obodnom} and 
\cite{gev-mov}) introduced the notion of a \emph{movable category} and proved criteria 
for the movability and strong movability of a topological space.

\begin{definition}\label{movcat}
Let $K$ and $L$ be any categories, and let $\Phi\colon K\to L$ be any covariant functor.
The category $K$ is said to be \emph{movable with respect to the category $L$ and
the functor $\Phi\colon K\to L$} if, given any object $X\in \Ob (K)$,
there is an object $Y\in \Ob(K)$ and a morphism  $f\in
\Mor_{K}(Y, X)$ such that, for any object $Z\in \Ob(K)$ and any
morphism $g\in \Mor_{K}(Z, X)$, there exists a morphism $h\in
\Mor_L(\Phi(Y), \Phi(Z))$ satisfying the condition $\Phi(g) h = \Phi(f)$.

If $K$ is a full subcategory of the category $L$ and $\Phi\colon K\hookrightarrow L$ is 
the embedding functor, then the subcategory $K$ is said to be \emph{movable with respect to 
the category $L$}.

If $K=L$ and $\Phi = 1_K$, then the category $K$ is called \emph{strongly movable}.
\end{definition}

\begin{theorem}[Avakyan and Gevorgyan \cite{av-gev}]
A topological space $X$ is movable if and only if the category $W^X$ is movable with respect to 
the category H-CW and forgetful functor $\Omega\colon W^X\to \text{H-CW}$. 
\end{theorem}

The \emph{forgetful functor} $\Omega\colon W^X\to \text{H-CW}$ takes each object $f\colon X\to Q$ 
to the object $Q\in \text{H-CW}$ and each morphism $\eta \colon  (f\colon X\to Q)\to (f'\colon 
X\to Q')$, $\eta \circ f = f'$, of the category $W^X$ to a morphism $\eta \colon  Q\to Q'$ 
of the category H-CW.

\begin{theorem}[Avakyan and Gevorgyan \cite{av-gev}]\label{th-main1}
The topological space $X$ is strongly movable if and only if so is the category $W^X$.
 \end{theorem}

Similar theorems for \emph{uniformly movable} spaces and categories were proved by Gevorgyan and 
Pop~\cite{Gev-Pop-4}.

\begin{definition}\label{u-movcat}
A category $K$ is said to be \emph{uniformly movable} if, given any object $X\in K$, 
there exists an object $M(X)\in K$ and a morphism  $m_X \colon  M(X) \to X$ satisfying the 
following conditions: 

(i) for any object $Y\in K$ and any morphism $p \colon  Y\to X$ of the category $K$, there exists 
a morphism $u(p) \colon  M(X) \to Y$ such that $p u(p) = m(X)$;

(ii) for any objects $Y, Z \in K$ and morphisms $p\colon Y \to X$, $q\colon  Z \to X$, and 
$r\colon Z\to Y$ of $K$ satisfying the condition $p  r = q$, the relation $r  u(q) = u(p)$ holds. 
\end{definition}

\begin{theorem}[Gevorgyan and Pop \cite{Gev-Pop-4}]
A topological space $X$ is uniformly movable if and only if so is the category $W^X$. 
\end{theorem}

\section{Stable Spaces and Shape Retracts}

A topological space $X$ is said to be \emph{stable} if it has the shape of an ANR. Stability in 
the shape category was first considered by Porter \cite{porter-stability} and 
systematically studied by Demers  \cite{demers}, Edwards and Geoghegan \cite{ed-geogh-shapes}, 
\cite{ed-geogh-stab}, \cite{ed-geogh-stability}, Dydak \cite{dydak-asimple}, \cite{dydak-on}, 
\cite{dydak-1}, Geoghegan and Lacher \cite{geoghegan-lacher}, Porter \cite{porter-stabilityI}, 
\cite{porter-stabilityII}, and other authors.

Stability is a shape invariant. The notions of stability and pointed stability are equivalent, 
i.e.,  the following theorem  is valid (see Dydak's paper \cite{dydak-1} and Geoghegan's paper 
\cite{geoghegan-elementary}).

\begin{theorem}\label{th-ust}
A pointed space  $(X,*)$ is stable if and only if so is $X$. 
\end{theorem}

\begin{theorem}\label{th-domust}
If a connected pointed space $(X,*)$ is shape dominated by a pointed ANR-space $(P,*)$, then 
$(X,*)$ is stable. 
\end{theorem}

This theorem was independently  proved by Demers \cite{demers} and by Edwards 
and Geoghegan \cite{ed-geogh-shapes}, \cite{ed-geogh-stab}, who used different methods.

The following theorem of Dydak \cite{dydak-on} shows that, for connected pointed spaces, the 
notions of stability and strong movability coincide (see also Watanabe's 
paper~\cite{watanabe-onstrong}).

\begin{theorem}
A connected space $(X,*)$ is stable if and only if it is strongly movable.
\end{theorem}

If a pointed space $(X,*)$ is stable, then the pro-groups $\pro-\pi_k(X,*)$ 
are isomorphic to the shape groups $\check{\pi}_k(X,*)$. Therefore, the stability of 
$(X,*)$ implies that of the pro-groups $\pro-\pi_k(X,*)$. The converse is true 
for finite-dimensional connected spaces. Thus, stability admits the following algebraic 
characterization.

\begin{theorem}\label{th-ust=pro-ust}
Let $(X,*)$ be a connected space of finite shape dimension $\sd X$. Then $(X,*)$ 
is stable if and only if so are all homotopy pro-groups $\pro-\pi_k(X,*)$. 
\end{theorem}

This theorem was first proved by Edwards and Geoghegan  \cite{ed-geogh-stab}, 
\cite{ed-geogh-stability}, who used fairly sophisticated abstract tools. 
Subsequently, elementary proofs of this theorem were found by Dydak \cite{dydak-1} and 
Geoghegan~\cite{geoghegan-elementary}.

Let $X$ be a subspace of a topological space $Y$. A shape morphism $R\colon Y\to X$ is called 
a \emph{shape retraction} if $R i = 1_X$, where $i\colon X \to Y$ is a shape embedding. 
The space $X$ is then called a \emph{shape retract} of $Y$. A subspace $X\subset Y$ is called 
a \emph{neighborhood shape retract} of $Y$ if $X$ is a shape retract of some neighborhood $U$ 
of $X$ in~$Y$.

Absolute (neighborhood) shape retracts in shape theory are defined by analogy with the 
corresponding notions for the class of metrizable spaces. We say that a metrizable space $X$ is 
an \emph{absolute (neighborhood) shape retract} in the class of metrizable spaces and write 
$X\in \text{ASR}$ (respectively, $X\in \text{ANSR}$) if, for any closed embedding 
of $X$ in a metrizable space $Y$, the space $X$ is a (neighborhood) shape retract of $Y$ (see 
papers  \cite{borsuk-rus} by Borsuk,  \cite{mard-retracts} and 
\cite{mard-strongly} by Marde\v{s}i\'{c},  \cite{segal-movcont} by Segal, and  
\cite{shostak} by Shostak). It is easy to see that these notions are shape invariants.

The notion of an absolute (neighborhood) shape retract in the class of compact metrizable spaces 
coincides with that of a \emph{fundamental absolute (neighborhood) retract}, abbreviated as 
FAR (respectively, FANR), which was introduced by Borsuk in~\cite{borsuk-rus}.

One of the first questions arising in shape theory is as follows: When does a space 
have the shape of a point? It turns out that this property characterizes absolute shape retracts.

\begin{theorem}[Marde\v{s}i\'{c} \cite{mard-retracts}]\label{th-mard-asr}
A metrizable space $X$ is an ASR if and only if $\sh(X)=0$.
\end{theorem}

For metrizable compact spaces, this theorem was proved by Borsuk \cite{borsuk-6}, and in 
the class of all $p$-paracompact spaces, by Shostak \cite{shostak}. For arbitrary topological 
spaces, Theorem~\ref{th-mard-asr} is not true (Godlewski \cite{godlewski-an}).

The triviality of the shape of a space $X$ can also be characterized in the language of 
associated ANR-systems as follows. 

\begin{theorem}
Let $\{X_\alpha, \mathbf{p}_{\alpha\alpha'}, A\}$ be an ANR-system  associated with a space $X$. 
Then $X$ has trivial shape if and only if, for each $\alpha\in A$, there exists 
an $\alpha'\geqslant \alpha$ such that the projection $p_{\alpha\alpha'}$ is homotopic to a 
constant map. 
\end{theorem}

 In the case of compact Hausdorff spaces, this theorem was proved by Marde\v{s}i\'{c} 
\cite{mard-retracts}. Of interest is also the following theorem (see Dydak's paper~\cite{dydak-1}).

\begin{theorem}
A space $X$ is an absolute shape retract if and only if $\sd X < \infty$ and $X$ is shape 
$n$-connected for some $n$. 
\end{theorem}

\begin{theorem}
A space $X$ is an absolute shape retract if and only if it is movable and shape $n$-connected for 
all $n$. 
\end{theorem}

Recall that a space $X$ is said to be \emph{shape $n$-connected} if $\pro-\pi_k(X,*)=0$ for all 
$k \leqslant n$. For compact metrizable spaces, the last theorem was proved by Bogatyi 
\cite{bogat-1} and Borsuk~\cite{borsuk-4}.

The notion of an ASR is remarkable in many respects. In particular, on this notion  the 
definition of a cell-like map is based, which makes it possible to answer 
the natural question of under what assumptions a map $f\colon X\to Y$ generates a shape isomorphism 
(see Theorems \ref{th-6} and~\ref{th-7}).

The class of ANSRs is a natural extension of the class of ANRs, and the global properties of 
ANSRs resemble to a certain extent those of ANRs. Clearly, each compact 
space having the shape of an ANR is an ANSR. Borsuk~\cite{borsuk-onseveral} asked the 
converse question: Is it true that any compact metrizable ANSR has the shape of a compact 
metrizable ANR? Edwards and Geoghegan \cite{ed-geogh-shapes} gave a negative answer to this 
question. They constructed an example of a two-dimensional metrizable continuum $X$ which is 
an ANSR but does not have the shape of a finite complex. This continuum $X$ cannot have 
the shape of a compact ANR, because, as is well known (West \cite{west-1}), any compact metrizable 
ANR has the homotopy type (and hence the shape) of a finite complex. Nevertheless, any plane 
compact ANSR has the shape of a compact ANR. Indeed, a compact set $X$ in the plane is 
an ANSR if and only if its Betti numbers $p_0(X)$ and $p_1(X)$ are finite. Therefore, the shape 
of a plane compact ANSR $X$ equals the shape of a finite union of pairwise disjoint 
planar graphs, and such unions are compact ANRs. Since one-dimensional ANSRs have 
the shape of a compact set in the plane (see \cite{trybulec-on}), we can say that 
Borsuk's problem has a positive solution for one-dimensional compact ANSRs. As shown by the 
example of Edwards and Geoghegan mentioned above, for two-dimensional compact ANSRs, the solution 
is negative.

The ANSRs are characterized by the following theorem due to Marde\v si\'c and 
Segal~\cite{mard-seg-4}.

\begin{theorem}\label{th-mardsegal}
A metrizable space $X$ is an ANSR if and only if it is shape dominated by an ANR~$P$.
\end{theorem}

Importantly,  in the case of compact metrizable spaces, the shape dominating 
ANR $P$ can be chosen compact in this theorem.

\begin{theorem}\label{th-ansr-comp}
A compact metrizable space $X$ is an ANSR if and only if it is shape dominated by a compact 
ANR~$P$. 
\end{theorem}

Theorem \ref{th-mardsegal} is also valid in the pointed case.

\begin{theorem}\label{th-mardsegalpunkt}
A metrizable space $(X,*)$ is an ANSR if and only if it is shape dominated by an ANR $(P,*)$. 
\end{theorem}

Theorems~\ref{th-ust}, \ref{th-domust}, and~\ref{th-mardsegalpunkt} imply the following important 
property of pointed ANSR (see~\cite{mard-seg-4}).

\begin{theorem}\label{th-ansr=ust}
A connected metrizable space $(X,*)$ is an ANSR if and only if $(X,*)$ is stable, i.e., has  the 
shape of an ANR $(P,*)$. 
\end{theorem}

For compact metrizable spaces, this theorem was proved by Edwards and Geoghegan 
\cite{ed-geogh-stab}, \cite{ed-geogh-stability}. In Theorem~\ref{th-ansr=ust} an ANR 
$(P,*)$ cannot be replaced by a compact ANR even in the case of a metrizable 
continuum. Indeed, the above-mentioned two-dimensional metrizable continuum constructed by Edwards 
and Geoghegan \cite{ed-geogh-shapes} is an ANSR but does not have the shape of a compact 
ANR. The question of when a compact metrizable space $(X,*)$ has the shape of a compact  
ANR has been answered by using tools of algebraic $K$-theory. To be more precise, 
a metrizable continuum $(X,*)$ has the shape of a finite complex precisely when 
the so-called Wall obstruction \cite{wall} $\sigma(X,*)$, which takes values in the reduced 
Grothendieck group $\widetilde{K}^0(\check{\pi}_1(X,*))$, is trivial (see~\cite{ed-geogh-shapes}).

Theorems~\ref{th-ust=pro-ust} and~\ref{th-ansr=ust} imply the following algebraic 
characterization of ANSRs.

\begin{theorem}\label{th-ansr=pro-ansr}
Let $(X,*)$ be a connected space of finite shape dimension $\sd X$. Then $(X,*)$ is 
an ANSR if and only if all homotopy pro-groups $\pro-\pi_k(X,*)$ are stable. 
\end{theorem}

The following important result was obtained by Hastings and Heller \cite{Hastings-heller}, 
\cite{Hastings-heller-split}.

\begin{theorem}\label{th-ansr=ansr*}
Let $X$ be a connected ANSR. Then $(X,*)$ is a pointed ANSR.
\end{theorem}

The proof of this theorem has become possible after the splitting problem for homotopy 
idempotents was solved. Recall that a map  $f\colon X \to X$ is called a \emph{homotopy 
idempotent} if $f^2\simeq f$. A map  $f\colon X \to X$ is said to \emph{split} in the category H-CW 
if there exists a CW-complex $P$ and maps $u\colon P \to X$ and $v\colon X\to P$ such that 
$v\circ u \simeq 1_P$ and $u\circ v \simeq f$.

\begin{theorem}
Any homotopy idempotent $f\colon P \to P$ of a finite-dimensional CW-complex $P$ splits.
\end{theorem}

In the case of infinite-dimensional CW-complexes, this statement is false. Dydak 
\cite{dydak-asimple} and Minc, and also Freyd and Heller  \cite{freyd-heller},  
independently proved the existence of a group $G$ and a homomorphism $\phi\colon G\to G$ 
which induces a nonsplitting homotopy idempotent $f\colon K(G,1)\to K(G,1)$ 
of the infinite-dimensional Eilenberg--MacLane complex $K(G,1)$. Moreover, the group $G$ is 
universal in the sense that if $f'\colon X\to X$ is a nonsplitting homotopy idempotent of an 
infinite-dimensional CW-complex $X$, then there exists an injection $G\to \pi_1(X)$ 
equivariant with respect to $f_*$ and~$f'_*$.

However, it is well known that any homotopy idempotent $f\colon (X,*)\to (Y,*)$ between 
pointed connected CW-complexes splits (see papers  \cite{brown} by Brown, 
\cite{ed-geogh-shapes} by Edwards and Geoghegan, and \cite{freyd} by~Freyd).

Note that, in the nonpointed case, the problem of splitting homotopy idempotents 
arises in various areas of topology. In homotopy theory this problem is closely related 
to Brown's theorem on the representability of half-exact functors (see papers 
\cite{freyd-heller} by Freyd and Heller and  \cite{heller-on} by Heller), and in the shape theory, 
to the study of ANSRs.

It follows from Theorems~\ref{th-ust} and~\ref{th-ansr=ansr*} that Theorem~\ref{th-ansr=ust}  
remains valid in the nonpointed case.

\begin{theorem}
A connected metrizable space $X$ is an ANSR if and only if $X$ is stable. 
\end{theorem}

\section{Whitehead's and Hurewicz' Theorems in Shape Theory}

It was mentioned in the introduction that Whitehead's classical theorem is not valid for 
arbitrary spaces instead of CW-complexes. Extending Whitehead's theorem to more 
general spaces is one of the objectives of shape theory.

Let $(X,*)$ be a pointed space, and let $(\mathbf{X}, \boldsymbol{*}) = ((X_\alpha, *), 
\mathbf{p}_{\alpha\alpha'}, A)$ be  an inverse ANR-system associated with it. Applying the functors 
$\pi_n$, $n=0,1, \dots $, to this system, we obtain  the \emph{homotopy pro-groups} of the 
pointed space $(X,*)$: 
$$
\pro-\pi_n(X,*) = (\pi_n(X_\alpha, *), \mathbf{p}_{\alpha\alpha'\#}, A).
$$ % 
The homotopy pro-groups $\pro-\pi_n(X,*)$ are an analogue of the homotopy groups 
$\pi_n(X,*)$ in shape theory.

Recall that any shape morphism $F\colon (X,*) \to (Y, *)$ is determined by a morphism 
$\mathbf{f} \colon (\mathbf{X}, \boldsymbol{*}) \to (\mathbf{Y}, \boldsymbol{*})$ in the category 
pro-H-CW; this suggest the natural definition of a \emph{morphism} 
$\pro-\pi_n(F)\colon \pro-\pi_n(X,*) \to \pro-\pi_n(Y,*)$ \emph{of homotopy pro-groups}.

It is easy to show that $\pro-\pi_n$ is a covariant functor from the shape category $\Sh-Top_*$ to 
the category  $\pro-\Set_*$ for $n=0$ and to the category $\pro-\Grp_*$ for $n=1,2, \dots$\,. 
In particular,  $\pro-\pi_n(X,*)$ is a shape invariant.

The inverse limits of the pro-homotopy groups $\pro-\pi_n(X,*)$ are called the \emph{shape groups} 
of the pointed space $(X,*)$ and denoted by $\check{\pi}_n(X,*)$: 
\[ \check{\pi}_n(X,*) = 
\lim_{\longleftarrow}(\pi_n(X_\alpha, *), p_{\alpha\alpha'\#}, A).
\] 
If $F\colon (X,*) \to (Y, *)$ is a shape morphism, then the passage to the limit yields 
a homomorphism  $\check{\pi}_n(F) \colon  \check{\pi}_n(X,*) \to \check{\pi}_n(Y,*)$ 
of shape groups. Therefore, $\check{\pi}_n$ is a covariant functor from the category $\Sh-Top_*$ 
to the category $\Set_*$ (for  $n=0$) or to the category $\Grp_*$ (for $n=1,2, \ldots$).

The \emph{homology pro-group} $\pro-H_n(X)$ is defined in a similar way:
\[
\pro-H_n(X) = (H_n(X_\alpha), H_n(p_{\alpha\alpha'}), A).
\]
Passing to the limit in this system, we obtain the well-known \emph{\v Cech homology group} 
$\check{H}_n(X)$ of the space~$X$.

It should be mentioned that, for spaces with good local structure, such as ANRs, 
the homotopy (homology) pro-groups can be replaced by shape groups (\v Cech homology groups). 
However, in the general case, a part of information is lost under the passage to the limit. 
Therefore, the homotopy or homology pro-groups  contain more information about 
the space $X$ than the shape or \v Cech homology groups.

The following theorem of Morita \cite{morita-2} is the most general version of Whitehead's theorem  
in shape theory.

\begin{theorem}[shape version of Whitehead's theorem, Morita \cite{morita-2}]\label{th-mor}
A shape morphism $F\colon  (X, *) \to  (Y, *)$ of finite-dimensional (in the sense of dimension sd) 
connected topological spaces is a shape equivalence if and only if the induced 
homomorphisms $\pro-\pi_n(F)\colon \pro-\pi_n(X,*) \to \pro-\pi_n(Y,*)$ of 
homotopy pro-groups are isomorphisms for all~$n$. 
\end{theorem}

Unlike Whitehead's classical theorem, this theorem is valid for any spaces, and it deals with 
homotopy pro-groups rather than homotopy groups. However, it involves a new constraint, namely, the 
finite dimensionality of the spaces $X$ and $Y$.  This constraint cannot be removed. 
The corresponding example was constructed in \cite{draper-keesling} by using the Kahn metric 
continuum \cite{kahn}, which is defined as follows. 
First, a space $X_0$ is defined as the CW-complex obtained by attaching a $(2p + 1)$-cell 
to the sphere $S^{2p}$ by a map of prime degree $p$. Then, for each $n \geqslant 0$, a space 
$X_{n+1}$ is constructed by induction as the $(2p - 2)$-fold suspension 
of $X_n$: 
\[ X_{n+1}=\Sigma^{2p-2}(X_n). \] 
Between these spaces maps $f_n \colon X_n \to X_{n-1}$, $n > 1$, are 
defined, also by induction, as 
\[ f_n = \Sigma^{2p-2}(f_{n-1}), 
\] 
beginning with map $f_1\colon  X_1 \to X_0$ constructed in a certain special way.
					    
The Kahn compact space $K$ is the limit of the inverse sequence
\[
X_0 \overset{f_1}\longleftarrow  X_1 \overset{f_2}\longleftarrow  X_2 \longleftarrow \cdots
\]
Since $\pi_k\left(\Sigma^{n(2p-2)}, *\right)=0$ for $n(2p-2) \geqslant k$, it follows that 
$\pro-\pi_k(K,*)=0$ for all $k$. However, $K$ does not have the shape of a point, 
because the composition 
\[ 
f_1\circ f_2 \circ \ldots \circ f_n \colon  X_n \to X_0 
\] 
is essential (i.e., not homotopic to a constant) for any $n$ 
(see Adams' papers \cite{adams-lect} and \cite{adams}). This fact follows also from Toda's 
results~\cite{toda-on}.

The first Whitehead-type theorem in shape theory was proved by Moszynska \cite{mosz}. She 
proved Theorem~\ref{th-mor} for metrizable continua $(X, *)$ and  $(Y, *)$. Marde\v{s}i\'{c}  
\cite{mard-0} proved this theorem in the case where $(X, *)$ and $(Y, *)$ are compact and  $(Y, 
*)$ is metrizable and in the case where $(X, *)$ and $(Y, *)$ are any spaces and the shape morphism 
$F\colon (X, *) \to  (Y, *)$ is induced by a continuous map.

For movable spaces, the shape analogue of Whitehead's theorem (Theorem~\ref{th-mor}) takes a 
simpler form \cite{mosz}, \cite{keesling-8}, \cite{dydak-some}, \cite{dydak-1}.

\begin{theorem}\label{th-witehead-mov}
Let  $(X,*)$ and $(Y,*)$ be movable metrizable continua of finite shape dimension. Then a shape 
morphism $F\colon  (X, *) \to  (Y, *)$  is a shape equivalence if and only if all induced 
homomorphisms $\check{\pi}_n(F) \colon  \check{\pi}_n(X,*) \to \check{\pi}_n(Y,*)$ of shape groups 
are isomorphisms. 
\end{theorem}

This theorem is a consequence of Theorem~\ref{th-mor} and the following result of 
independent interest  \cite{keesling-8}, \cite{dydak-1}.

\begin{theorem}\label{th-mov-iso}
Let  $F\colon  (X, *) \to  (Y, *)$ be a shape morphism movable metrizable continua. If, for 
some $n$, the induced homomorphism $\check{\pi}_n(F) \colon  \check{\pi}_n(X,*) \to 
\check{\pi}_n(Y,*)$ is an isomorphism (epimorphism) of shape groups, then $\pro-\pi_n(F)\colon 
\pro-\pi_n(X,*) \to \pro-\pi_n(Y,*)$ is an isomorphism (epimorphism) of pro-groups. 
\end{theorem}

The first ``infinite-dimensional'' Whitehead-type theorems were proved by Edwards and Geoghegan 
in~\cite{ed-geogh} and generalized in \cite{dydak-on} and~\cite{dydak-some}. The most general 
results are the following two theorems of Dydak~\cite{dydak-1}.

\begin{theorem}\label{th-dydak-0}
Let $(X, *)$ and $(Y, *)$ be connected spaces. Suppose that $(X, *)$ is movable 
and $F\colon (X, *) \to (Y, *)$ is a shape domination. If 
$\pro-\pi_n(F)\colon \pro-\pi_n(X,*) \to \pro-\pi_n(Y,*)$ is an isomorphism for every $n$, then $F$  
is a shape equivalence. 
\end{theorem}

\begin{theorem}\label{th-dydak}
Let $F\colon (X, *) \to (Y, *)$ be a shape morphism of connected spaces. Suppose that one of 
the following two conditions holds:

(i) $\sd X< \infty$ and $(Y, *)$ movable;

(ii)  $(X, *)$ is movable and $\sd Y< \infty$.

\noindent
If $\pro-\pi_n(F)\colon \pro-\pi_n(X,*) \to \pro-\pi_n(Y,*)$ is an isomorphism for every $n$, 
then $F$ is a shape equivalence.
\end{theorem}

In particular, if $(X, *)$ is movable and all pro-groups $\pro-\pi_n(X,*)$ are trivial, then 
$\sh(X,*)=0$. In the last theorem the finite-dimensionality of spaces cannot be 
replaced by their movability. Draper and Keesling \cite{draper-keesling} constructed a map  
$f\colon (X,*) \to(Y,*)$ of movable metrizable continua which induces isomorphisms of all homotopy 
pro-groups but is not a shape equivalence. A similar example was also constructed by 
Kozlowski and Segal \cite{kozlowski-segal-0}.

A shape analogue of Whitehead's theorem for movable morphisms was obtained by Gevorgyan 
and Pop \cite{Gev-Pop-1}.

\begin{theorem}\label{th-gev-pop-1}
Let $(X, *)$ and $(Y, *)$ be connected spaces. Suppose that $\sd X < \infty$ and $F\colon (X, *) 
\to (Y, *)$ is a movable shape morphism being a shape domination. If $\pro-\pi_n(F)\colon 
\pro-\pi_n(X,*) \to \pro-\pi_n(Y,*)$ is an isomorphism for every $n$, then $F$ is a shape 
equivalence. 
\end{theorem}

Whitehead's theorem has also homology versions (see \cite{mard-01}, 
\cite{morita-2}, \cite{raussen}). In their proofs the key role is played by the following result 
(see \cite{morita-2}, \cite{raussen}).

\begin{theorem}
Let  $F\colon (X, *) \to (Y, *)$ be a shape morphism. If 
$$
\pro-\pi_0(X,*)=\pro-\pi_1(X,*)=\pro-\pi_0(Y,*)=\pro-\pi_1(Y,*)=0,
$$ 
then 
the following two conditions are equivalent for each $n\geqslant 2$:

(i) $\pro-\pi_k(F)$ is an isomorphism for $k<n$ and an epimorphism for $k=n$;

(ii) $\pro-H_k(F)$ is an isomorphism for $k<n$ and an epimorphism for $k=n$.
\end{theorem}

The well-known \emph{Hurewicz homomorphism}  $\phi_n\colon \pi_n(X,*) \to H_n(X)$ naturally 
generates a morphism $\pro-\phi_n\colon  \pro-\pi_n(X, *) \to \pro-H_n(X)$. The passage to the 
limit yields also the \emph{spectral Hurewicz homomorphism} $\check{\phi}_n \colon  
\check{\pi}_n(X,*) \to \check{H}_n(X)$.

Hurewicz' classical theorem asserts that \emph{if $(X,*)$ is an $(n-1)$-connected space, i.e., 
$\pi_k(X,*)=0$ for all $0\leqslant k\leqslant n-1$, then, for any $n\geqslant 2$,  the 
following conditions hold: 
\textup{(i)}~$H_k(X)=0$, $0\leqslant k\leqslant n-1$; \textup{(ii)}~$\phi_n\colon \pi_n(X,*) \to 
H_n(X)$ is an isomorphism; \textup{(iii)}~$\phi_{n+1}$ is an epimorphism; \textup{(iv)}~$\phi_1\colon \pi_1(X,*) \to 
H_1(X)$ is an epimorphism.}

This theorem was transferred to shape theory in various forms by many authors. The 
following version of Hurewicz' theorem for the pro-category can be found in~\cite{mard-seg-4}.

\begin{theorem}
Let $(X, *)$ be a shape $(n-1)$-connected space, i.e., $\pro-\pi_k(X,*)=0$, $k \leqslant n-1$. 
Then, for all $n \geqslant 2$, the following conditions hold:

(i) $\pro-H_k(X) = 0$, $1\leqslant k \leqslant n-1$;

(ii) $\pro-\phi_n\colon  \pro-\pi_n(X, *) \to \pro-H_n(X)$ is an isomorphism of pro-groups;

(iii) $\pro-\phi_{n+1}\colon  \pro-\pi_{n+1}(X, *) \to \pro-H_{n+1}(X)$ is an epimorphism 
if pro-groups.

\noindent
For $n=1$, 

(iv) $\pro-\phi_1\colon  \pro-\pi_1(X, *) \to \pro-H_1(X)$ is an epimorphism.
\end{theorem}

Somewhat different statements of this theorem were presented in  \cite{artin-mazur}, 
\cite{mard-ungar}, and \cite{raussen}.

Kuperberg \cite{kuper-1} proved the following theorem, which is a special case of a more general  
result due to Artin and Mazur~\cite{artin-mazur}.

\begin{theorem}\label{th-kuperberg}
If a compact metrizable space $(X, *)$ is shape $(n-1)$-connected for $n\geqslant 2$, then 
the Hurewicz spectral homomorphism $\check{\phi}_n \colon  \check{\pi}_n(X,*) \to \check{H}_n(X)$ 
is an isomorphism. 
\end{theorem}

For movable compact metrizable spaces $(X,*)$, the assumption that $(X,*)$ is shape 
$(n-1)$-connected can be replaced by the weaker assumption $\check{\pi}_{k}(X, *) = 0$, 
$k \leqslant n-1$ (see~\cite{kuper-1}).

\begin{theorem}\label{th-hur-mov}
Let $(X, *)$ be a movable compact metrizable space such that $\check{\pi}_k(X, *) = 0$ for 
all $k \leqslant n-1$, $n \geqslant 2$. Then $\check{\phi}_n \colon  \check{\pi}_n(X,*) \to 
\check{H}_n(X)$ is an isomorphism. 
\end{theorem}

The movability assumption in this theorem is essential. Indeed, let $X$ be the 2-fold suspension 
of the 3-adic solenoid. The space $X$ is connected, and $\check{\pi}_k(X, *) = 0$ for all $k 
\leqslant 3$. However,  $\check{\pi}_4(X, *)$ is not isomorphic to $\check{H}_4(X)$, because 
$\check{\pi}_4(X, *) \neq 0$, while $\check{H}_4(X)=0$ (see~\cite{kuper-anote}).

The metrizability assumption in Theorem~\ref{th-hur-mov} is essential as well. 
Keesling \cite{keesling-12} proved that, for any $n\geqslant 2$, there exists a nonmetrizable 
movable continuum $(X,*)$ such that $\check{\pi}_k(X, *) = 0$ for all $k \leqslant n-1$ 
but $\check{\pi}_n(X, *)$ is not isomorphic to $\check{H}_n(X)$.

In papers  \cite{porter-achech} by Porter,  \cite{mard-ungar} by Marde\v{s}i\'{c} and Ungar, and 
 \cite{ungar} by Ungar Theorem~\ref{th-kuperberg} was proved for pairs $(X, A)$ (see also 
\cite{morita-2} and \cite{watanabe-anote}). A Hurewicz theorem with an application of Steenrod 
homology is due to Kodama and Koyama~\cite{kodama-koyama}.

It should be mentioned that the first spectral Hurewicz theorem was proved by Christie 
\cite{christie} but only in strong shape theory. Subsequently, similar 
theorems with an application of Steenrod homology were proved in \cite{lisica}, 
\cite{kodama-koyama}, and \cite{quigley}.

\section{The Shape Dimension of Spaces and Maps}

The dimension of topological spaces is not a shape invariant. For example, a contractible space, 
which may have arbitrarily large dimension,  has the shape of a point, which is zero-dimensional. 
This leads to the necessity of introducing a numerical shape 
invariant playing the same role in shape theory as dimension plays in 
topology. Such an invariant, called \emph{fundamental dimension}, was introduced by Borsuk 
\cite{borsuk-2}, who defined it first for compact metrizable spaces.

\begin{definition}\label{def-4}
The \emph{fundamental dimension} of a compact metrizable space $X$, denoted by $\Fd X$, is 
defined as the least dimension $\dim Y$ of a compact metrizable space $Y$ such that 
$\sh Y \geqslant \sh X$: 
\[ 
\Fd X= \min \{\dim Y: \sh Y \geqslant \sh X\}.
\] 
\end{definition}

The fundamental dimension of compact metrizable spaces was thoroughly studied by Nov\'ak 
\cite{nowak-0, nowak-2, nowak, nowak-4, nowak-5, nowak-1, nowak-3} and  
Spie\.z~\cite{spiez-1,spiez-2}.

For arbitrary topological spaces, Dydak \cite{dydak} introduced the notion of \emph{deformation 
dimension}; its definition in a slightly modified form is as follows. 

\begin{definition}
A topological space $X$ has \emph{deformation dimension} $\ddim X \leqslant n$ if any map $f\colon 
X\to P$ to an ANR $P$ can be homotopically factored through an ANR $P'$ of dimension $\dim 
P'\leqslant n$, i.e., $f\simeq h\circ g$, where $g\colon X\to P'$ and $h\colon P' \to P$ are 
some maps. 
\end{definition}

Marde\v{s}i\'{c} and Segal \cite{mard-seg-4} defined \emph{shape dimension} as follows.

\begin{definition}\label{def-2}
A topological space $X$ has \emph{shape dimension} $\sd X\leqslant n$ if there exists an associated 
inverse ANR-system $\mathbf{X}=\{ X_\alpha , \mathbf{p}_{\alpha\alpha '}, A \}$ such that 
each $X_\alpha$, $\alpha \in A$, is homotopy dominated by an ANR of dimension $ \leqslant n$.

If $\sd X\leqslant n$ and $n$ is the least such number, then we say that $X$ has \emph{shape 
dimension $n$} and write $\sd X = n$.

If $\sd X\leqslant n$ for no integer $n\geqslant 0$, then we say that $X$ 
has \emph{infinite shape dimension} and write $\sd X = \infty$. 
\end{definition}

The following theorem is valid  \cite{dydak}.

\begin{theorem}\label{th-1}
A topological space $X$ has shape dimension $\sd X\leqslant n$ if and only if there exists 
an associated inverse ANR-system $\mathbf{X}=\{ X_\alpha , \mathbf{p}_{\alpha\alpha 
'}, A \}$ such that $\dim X_\alpha \leqslant n$ for all $\alpha \in A$. 
\end{theorem}

\begin{proof}
Let $\sd X\leqslant n$. Consider an associated inverse ANR-system $\mathbf{X}=\{ X_\alpha , 
\mathbf{p}_{\alpha\alpha '}, A \}$ satisfying the condition in Definition~\ref{def-2}. This means 
that, for each space $X_\alpha$, there exists an ANR $Y_\alpha$ of dimension $\dim Y_\alpha 
\leqslant n$ and homotopy classes $\mathbf{f}_\alpha \colon  X_\alpha \to Y_\alpha$ and 
$\mathbf{g}_\alpha \colon  Y_\alpha \to X_\alpha$ such that
\begin{equation}\label{eq_1}
\mathbf{g}_\alpha \mathbf{f}_\alpha = \mathbf{1}_{X_\alpha}.
\end{equation}

For any $\alpha, \alpha ' \in A$, $\alpha \leqslant \alpha '$, we define a homotopy class 
$\mathbf{q}_{\alpha\alpha'}\colon Y_{\alpha '} \to Y_{\alpha}$ by 
\begin{equation}\label{eq_2}
\mathbf{q}_{\alpha\alpha'} = \mathbf{f}_\alpha  \mathbf{p}_{\alpha\alpha '} \mathbf{g}_{\alpha '}.
\end{equation}
Note that if $\alpha \leqslant \alpha ' \leqslant \alpha''$, then
\begin{equation}
\mathbf{q}_{\alpha\alpha'} \mathbf{q}_{\alpha'\alpha''} = \mathbf{q}_{\alpha\alpha''}.
\end{equation}
Indeed, according to \eqref{eq_1} and \eqref{eq_2}, we have
\[
\mathbf{q}_{\alpha\alpha'} \mathbf{q}_{\alpha'\alpha''} = \mathbf{f}_\alpha \mathbf{p}_{\alpha\alpha '} \mathbf{g}_{\alpha '}\mathbf{f}_\alpha' \mathbf{p}_{\alpha'\alpha ''} \mathbf{g}_{\alpha ''}=
\mathbf{f}_\alpha \mathbf{p}_{\alpha\alpha '} \mathbf{1}_{X_{\alpha'}} \mathbf{p}_{\alpha'\alpha ''} \mathbf{g}_{\alpha ''} = \mathbf{f}_\alpha \mathbf{p}_{\alpha\alpha ''} \mathbf{g}_{\alpha ''} = \mathbf{q}_{\alpha\alpha''}.
\]

Thus, we have constructed an inverse ANR $\mathbf{Y}=\{ Y_\alpha , \mathbf{q}_{\alpha\alpha 
'}, A \}$, where $\dim Y_\alpha \leqslant n$ for all $\alpha \in A$.  Let us prove that this 
inverse system is associated with the space $X$ (see Definition~\ref{def-3}).

Let $\alpha \in A$ be any index. Consider the homotopy class $\mathbf{q}_\alpha \colon  X \to 
Y_\alpha$ defined by
\begin{equation}\label{eq_3} 
\mathbf{q}_{\alpha} = \mathbf{f}_\alpha \mathbf{p}_{\alpha}. 
\end{equation} 
This homotopy class satisfies the condition 
\begin{equation}\label{eq_4}
\mathbf{q}_{\alpha \alpha '} \mathbf{q}_{\alpha'} = \mathbf{q}_{\alpha},
\end{equation}
where $\alpha \leqslant \alpha '$.
Indeed, by virtue of \eqref{eq_1}, \eqref{eq_2}, and \eqref{eq_3} we have
\[
\mathbf{q}_{\alpha \alpha '} \mathbf{q}_{\alpha'} = \mathbf{f}_\alpha \mathbf{p}_{\alpha\alpha '} \mathbf{g}_{\alpha '}   \mathbf{f}_{\alpha'} \mathbf{p}_{\alpha'} = \mathbf{f}_\alpha \mathbf{p}_{\alpha\alpha '} \mathbf{1}_{X_{\alpha'}}  \mathbf{p}_{\alpha'} = \mathbf{f}_\alpha \mathbf{p}_{\alpha} = \mathbf{q}_{\alpha}.
\]

Now, let us verify condition (ii) in Definition~\ref{def-3}. Let  $P$ be any ANR, and let 
$\mathbf{f}\colon X\to P$ be any homotopy class. Then there exists an index $\alpha \in A$ and 
a homotopy class $\mathbf{h}_\alpha \colon  X_\alpha \to P$ such that 
\begin{equation}\label{eq-6} 
\mathbf{h}_\alpha \mathbf{p}_\alpha = \mathbf{f}. 
\end{equation} 
Consider $\tilde{\mathbf{h}}_\alpha\colon Y_\alpha \to P$ defined by 
\begin{equation}\label{eq-5} 
\tilde{\mathbf{h}}_\alpha = \mathbf{h}_\alpha \mathbf{g}_\alpha. 
\end{equation} 
We claim that $\tilde{\mathbf{h}}_\alpha \mathbf{q}_\alpha = \mathbf{f}$. 
Indeed, taking into account \eqref{eq_1}, \eqref{eq_3}, \eqref {eq-6}, and \eqref {eq-5}, we obtain 
\[ 
\tilde{\mathbf{h}}_\alpha \mathbf{q}_\alpha = \mathbf{h}_\alpha \mathbf{g}_\alpha 
\mathbf{f}_\alpha \mathbf{p}_{\alpha} = \mathbf{h}_\alpha \mathbf{1}_{X_\alpha} \mathbf{p}_{\alpha} 
= \mathbf{h}_\alpha \mathbf{p}_{\alpha} = \mathbf{f}. 
\]

It remains to verify condition (iii) in Definition~\ref{def-3}. Suppose that 
$\boldsymbol{\varphi}, \boldsymbol{\psi} \colon Y_\alpha \to P$ are  homotopy classes such that 
\begin{equation} 
\boldsymbol{\varphi} \mathbf{q}_\alpha = \boldsymbol{\psi} \mathbf{q}_\alpha.
\end{equation}
By virtue of~\eqref{eq_3} this relation implies 
\begin{equation}
\boldsymbol{\varphi} \mathbf{f}_\alpha \mathbf{p}_\alpha = \boldsymbol{\psi} \mathbf{f}_\alpha \mathbf{p}_\alpha.
\end{equation}
Hence there exists an $\alpha' \in A$, $\alpha '  \geqslant \alpha$, for which 
\begin{equation}\label{eq-7}
\boldsymbol{\varphi} \mathbf{f}_\alpha \mathbf{p}_{\alpha\alpha'} = \boldsymbol{\psi} \mathbf{f}_\alpha \mathbf{p}_{\alpha\alpha'}.
\end{equation}
Relations \eqref{eq-2} and \eqref{eq-7} directly imply
\[
\boldsymbol{\varphi} \mathbf{q}_{\alpha\alpha'} =  \boldsymbol{\varphi} \mathbf{f}_\alpha \mathbf{p}_{\alpha\alpha '} \mathbf{g}_{\alpha '} =  \boldsymbol{\varphi} \mathbf{f}_\alpha \mathbf{p}_{\alpha\alpha'} \mathbf{g}_{\alpha '} = \boldsymbol{\psi} \mathbf{f}_\alpha \mathbf{p}_{\alpha\alpha'} \mathbf{g}_{\alpha '} =
\boldsymbol{\psi} \mathbf{f}_\alpha \mathbf{p}_{\alpha\alpha '} \mathbf{g}_{\alpha '} = \boldsymbol{\psi} \mathbf{q}_{\alpha\alpha'}.
\]

The sufficiency part of the theorem is obvious.
\end{proof}

The following theorem \cite{mard-seg-4} shows that, for any topological space, the notions 
of deformation dimension and shape dimension are equivalent.

\begin{theorem}\label{th-2}
Let  $X$ be topological space. Then $\sd X \leqslant n$ if and only if $\ddim X \leqslant n$.
\end{theorem}

\begin{proof}
The implication $\sd X \leqslant n \implies \ddim X \leqslant n$ easily follows from 
Theorem~\ref{th-1}.

Let us prove the implication $\ddim X \leqslant n \implies \sd X \leqslant n$. Consider 
an inverse ANR-system $\{ X_\alpha , \mathbf{p}_{\alpha\alpha '}, A \}$ associated with $X$. 
In view of the  CW-approximation theorem,  we can 
assume without loss of generality that all $X_\alpha$ are CW-complexes and all  
$p_{\alpha\alpha '}$ are cell maps.

Since $\ddim X \leqslant n$, it follows that, for each map $p_\alpha \colon  X\to X_\alpha$, there 
exists a CW-complex $P_\alpha$ with $\dim P_\alpha \leqslant n$ and maps $u_\alpha \colon  
X\to P_\alpha$ and $v_\alpha \colon  P_\alpha \to X_\alpha$ such that $p_\alpha \simeq v_\alpha 
\circ u_\alpha$. In view of the CW-approximation theorem, we can also assume 
that $v_\alpha \colon  P_\alpha \to X_\alpha$ is a cell map, i.e., a map to the $n$-skeleton  
$X_\alpha^n \subset X_\alpha$.

Now, consider the inverse system $\{ X_\alpha^n, \mathbf{q}_{\alpha\alpha '}, A \}$, where  
each $X_\alpha^n$ is the $n$-skeleton of the CW-complex $X_\alpha$ and 
$q_{\alpha\alpha '}=p_{\alpha\alpha '}|X_{\alpha'}^n$. For each $\alpha \in A$, we define a 
map $q_\alpha \colon  X \to X_\alpha^n$ by $q_\alpha = v_\alpha  u_\alpha$. It is easy to 
show that this inverse ANR-system is associated with the space $X$. To complete the proof, it 
remains to apply Theorem~\ref{th-1}. 
\end{proof}

The following theorem is a corollary of Theorem~\ref{th-1}.

\begin{theorem}
For any topological space $X$, $\sd X \leqslant \dim X$.
\end{theorem}

\begin{proof}
Let  $\dim X \leqslant n$, and let $\{ X_\alpha , \mathbf{p}_{\alpha\alpha '}, A \}$ be an 
associated inverse system consisting of CW-complexes $X_\alpha$. In view of  the 
 CW-approximation theorem, we can assume that all  $p_{\alpha\alpha '}$ are cell maps, 
i.e., $p_{\alpha\alpha '}(X_{\alpha'}^n)\subset X_\alpha^n$, where $X_\alpha^n$ is the 
$n$-skeleton of the CW-complex $X_\alpha$. Consider the inverse system $\{ X_\alpha^n , 
\mathbf{q}_{\alpha\alpha '}, A \}$, where $q_{\alpha\alpha'}=p_{\alpha\alpha '}|X_{\alpha'}$.  
Since $\dim X \leqslant n$, it follows that there exists a map $q_\alpha \colon  X \to X_\alpha^n$ 
such that $i\circ q_\alpha \simeq p_\alpha$, where $i\colon X_\alpha^n \to X_\alpha$ is the 
inclusion map. It is easy to prove that this system with limit projections $q_\alpha 
\colon X \to X_\alpha^n$, $\alpha \in A$, is associated with the space $X$. Now, applying 
Theorem~\ref{th-1}, we obtain $\sd X \leqslant n$. 
\end{proof}

\begin{theorem}\label{th-4}
If $X$ and $Y$ are any spaces and $\sh X \leqslant \sh Y$, then $\sd X \leqslant \sd Y$. 
\end{theorem}

\begin{proof}
Suppose that $\sh X \leqslant \sh Y$ and $\sd Y \leqslant n$. Consider any map $f\colon X \to P$ 
to an ANR $P$. By assumption there exist shape morphisms   $F\colon X \to Y$ and 
$G\colon Y \to X$ such that $G \circ F = 1_X$. To the map $f\colon X \to P$ the shape morphism $G$  
assigns a map $G(f)\colon Y \to P$ (see Definition~\ref{def-4}). 
Since $\sd Y \leqslant n$, it follows that there exists an ANR $P'$ with $\dim P' \leqslant 
n$ and maps $u\colon Y \to P'$ and $v\colon P' \to P$ such that $v \circ u \simeq G(f)$. But then 
$v \circ F(u) \simeq F(G(f))=f$. Therefore, according to Theorem~\ref{th-2}, we have $\sd X 
\leqslant n$. 
\end{proof}

\begin{corollary}\label{cor-1}
If $\sh X = \sh Y$, then $\sd X = \sd Y$.
\end{corollary}

Thus, the dimension  $\sd $ is a shape invariant.

It follows from Corollary~\ref{cor-1}  that the shape dimension of a contractible space 
equals $0$. Thus, the shape dimension of the disk equals $0$. On the other hand, the shape 
dimension of the circle $S^1$ equals $1$. Therefore, the shape dimension $\sd X$ is not a monotone 
function of $X$. However, as  Theorem~\ref{th-4} show, it is a monotone function of the shape 
$\sh X$.

\begin{theorem}[Holszty\'{n}ski, see \cite{nowak-1}]\label{th-3}
Let $X$ be a compact metrizable space with $\sd X = n$. Then there exists a compact metrizable 
space $Y$ with $\dim Y =n$ such that $\sh X = \sh Y$. 
\end{theorem}

\begin{corollary}
Let $X$ be a compact metrizable space. Then $\sd X \leqslant n$ if and only if there exists a 
compact metrizable space $Y$ such that $\sh X \leqslant \sh Y$ and $\dim Y \leqslant n$. 
\end{corollary}

It  immediately follows from this corollary that $\sd X = \Fd X$ for any compact metrizable space 
$X$, i.e.,  Borsuk's notion of fundamental dimension (see Definition~\ref{def-4}) coincides with 
Marde\v si\'c's shape dimension (see Definition~\ref{def-2}).

\begin{remark}
Theorem \ref{th-3} shows that the condition $\sh Y 
\geqslant \sh X$ in the  definition of fundamental dimension (Definition~\ref{def-4}) can be 
replaced by $\sh Y = \sh X$. 
\end{remark}

The following theorem is due to Borsuk \cite{borsuk-2}.

\begin{theorem}
Let $X$ and $Y$ be compact metrizable spaces. Then
\begin{equation}\label{eq-8}
\sd (X\times Y)\leqslant \sd X + \sd Y.
\end{equation}
\end{theorem}

Nov\'ak \cite{nowak} showed that there exist compact connected polyhedra $P$ and $Q$ with $\dim P 
=m$ and $\dim Q = n$, where $m,n \geqslant 3$, such that $\dim (P\times Q)= \max \{m,n\}$. It 
follows that inequality~\eqref{eq-8} cannot  generally be 
replaced by the equality 
\begin{equation}\label{eq-9} 
\sd (X\times Y) = \sd X + \sd Y. 
\end{equation} 
Spie\.z \cite{spiez-1}, \cite{spiez-2} proved that there exists a metrizable 
continuum $X$ of shape dimension $\sd X=2$ such that 
\[ \sd (X\times S^n)< \sd X +n. \] 
In particular, for $n=1$, we obtain 
\begin{equation}\label{eq-10} 
\sd (X\times S^1) = \sd X =2. 
\end{equation} 
Moreover, if a continuum $X$ satisfies condition~\eqref{eq-10}, then 
\begin{equation} \sd (X\times Y) = \sd (S^1\times Y)
\end{equation}
for any continuum $Y$ with $\sd Y> 0$.

Nevertheless, the following theorem of Nov\'ak is valid \cite{nowak-3}.

\begin{theorem}
Let $X$ and $Y$ be compact metrizable spaces such that $\sd X \neq 2$, $\dim Y = n$,  and the  
cohomology group $H^n(Y,G)$ is nontrivial for any nontrivial Abelian group $G$. Then 
equality~\eqref{eq-9} holds. 
\end{theorem}

Keesling  \cite{keesling-7} studied the shape dimension of topological groups; in particular, 
he proved the following theorem.

\begin{theorem}
If $X$ is a connected compact  group, then $\sd X = \dim X$.
\end{theorem}

The shape dimension of the Stone--\v Cech compactification $\beta X$ was studied in 
\cite{keesling-11}. We mention the following result.

\begin{theorem}
Let $X$ be a Lindel\"of space, and let $K\subset \beta X \backslash  X$ be a continuum 
in the Stone--\v Cech remainder of $X$. Then $\sd K = \dim K$. 
\end{theorem}

In relation to this theorem, we give the following interesting result of Winslow \cite{winslow}.

\begin{theorem}
The Stone--\v Cech remainder $\beta \mathbf{R}^3\backslash \mathbf{R}^3$ contains $\mathfrak{c}=2^{\aleph_0}$ 
continua of different shapes and $2^\mathfrak{c}$ continua of different topological types. 
\end{theorem}

The notion of the shape dimension of shape morphisms and continuous maps between topological spaces 
was introduced in~\cite{Gev-Pop-3}.

We say that the \emph{dimension $\dim (\mathbf{f}_\beta,\varphi)$} of a morphism 
$(\mathbf{f}_\beta,\varphi)\colon \mathbf{X}= (X_\alpha, \mathbf{p}_{\alpha\alpha'}, A)\rightarrow 
\mathbf{Y}=(Y_\beta,\mathbf{q}_{\beta\beta'}, B)$ of inverse ANR-systems is $\leqslant n$ if, for 
any $\beta\in B$, there exists an $\alpha\geqslant \varphi(\beta)$ such that the map 
$f_{\beta\alpha}\colon =f_\beta p_{\varphi(\beta)\alpha}\colon X_\alpha\rightarrow Y_\beta$ 
can be homotopically factored through an ANR $P$ of dimension $\dim P\leqslant n$, i.e., 
there exist maps $u\colon X_\alpha\rightarrow P$ and $v\colon P\rightarrow Y_\beta$ such that 
$f_{\beta\alpha} \simeq v\circ u$. This property is preserved under the passage to equivalent 
morphisms of inverse systems. Therefore, we can define the \emph{dimension of morphisms 
in the category pro-H-CW} and the \emph{shape dimension} of a shape morphism. We say that the  
\emph{shape dimension $\sd F$} of a shape morphism $F\colon X\to Y$ is $\leqslant n$ if 
the corresponding morphism $\mathbf{f}\colon \mathbf{X} \to \mathbf{Y}$ in the category pro-H-CW 
has dimension $\dim \mathbf{f}\leqslant n$. For a continuous map $f\colon X\to Y$, the condition 
$\sd f \leqslant n$ means that $\sd S(f)\leqslant n$.

\begin{theorem}\label{cor.2.13}
A space $X$ has shape dimension $\sd X\leqslant n$ if and only if the identity map $1_X$ has shape 
dimension $\sd 1_X\leqslant n$. 
\end{theorem}

\begin{theorem}\label{th-5}
Let $F\colon X\rightarrow Y$ be a shape morphism. Then $\sd F\leqslant \min (\sd X, \sd Y)$.
\end{theorem}

In particular, for a continuous map $f\colon X\rightarrow Y$, we have $\sd f\leqslant \min (\sd X, 
\sd Y)$.

It follows from Theorem~\ref{th-5} that maps between shape finite-dimensional spaces have finite 
shape dimension. There arises the natural question of whether the image of a shape 
finite-dimensional space  under a shape finite-dimensional map can be shape 
infinite-dimensional. In \cite{Gev-Pop-3}, this question was answered in the affirmative,  
and an example of a shape finite-dimensional surjective map $f\colon X\to 
Y$ between shape infinite-dimensional spaces $X$ and $Y$ was constructed.

\begin{theorem}\label{cor.2.15}
Let $F\colon X\rightarrow Y$ be a shape equivalence of topological spaces with $\sd F=n$. 
Then $\sd X=\sd Y=n$. 
\end{theorem}

The following theorem gives an important characterization of the shape dimension of a map $f\colon 
X\to Y$.

\begin{theorem}\label{th_2}
A map $f\colon X\rightarrow Y$ has shape dimension $\sd f\leqslant n$ if and only if, given any map 
$h\colon Y\rightarrow Q$ to an ANR-space $Q$, the composition $h\circ f$ can be homotopically  
factored through an ANR $P$ of dimension $\dim P\leqslant n$, i.e., there exist maps 
$u\colon X\rightarrow P$ and $v\colon P\rightarrow Q$ such that $h\circ f \simeq v\circ u$. 
\end{theorem}

\begin{theorem}\label{cor.3.5}
Let $F\colon X\rightarrow Y$ be a shape morphism of dimension $\sd F\leq n$. Then, for any Abelian 
group $G$ and any  $k>n$, the homomorphisms $\check{F}_k\colon \check{H}_k(X;G)\rightarrow 
\check{H}_k(Y;G)$ are trivial. 
\end{theorem}

\section{Embeddings in Shape Theory}

The notion of a \emph{shape embedding} is defined by analogy with topological embeddings and 
embeddings up to homotopy type. We say that a compact metrizable space $X$ is \emph{shape embedded} 
in a space $Y$ if there exists a compact metrizable subspace $X'$ of $Y$ for which $\sh X = \sh 
X'$.

First, we recall a classical embedding theorem for finite-dimensional compact metrizable spaces 
(see~\cite{alekc-pas}).

\begin{theorem}[Menger-N\"{o}beling-Pontryagin]
Any $n$-dimensional compact metrizable space $X$ is embedded in Euclidean space $\mathbf{R}^{2n+1}$.
\end{theorem}

Much  effort has been devoted to reducing the dimension of the ambient 
Euclidean space $\mathbf{R}^{2n+1}$. Flores \cite{flores} proved that this cannot be done without 
additional assumptions, having constructed  an example of a compact 
$n$-dimensional polyhedron not embedded in $\mathbf{R}^{2n}$ for each $n$. However, any 
smooth $n$-manifold is smoothly embedded in $\mathbf{R}^{2n}$ (see Whitney's paper~\cite{whitney}).

Below we recall the well-known theorem of Stallings on homotopy embeddings of finite 
CW-complexes, which has turned out to be very useful in the study of shape embeddings. A simple  
proof of Stallings' theorem can be found in Dranishnikov and Repov\v{s}' paper~\cite{dran-repovs}.

\begin{theorem}[Stallings \cite{stallings}]\label{th-stallings}
Any finite $n$-dimensional ($n>0$) CW-complex $K$ is embedded in $\mathbf{R}^{2n}$ up to homotopy 
type, i.e., there exists a polyhedron $M$ having the homotopy type of the CW-complex $K$ which is 
embedded in $\mathbf{R}^{2n}$. 
\end{theorem}

In this theorem the dimension of $\mathbf{R}^{2n}$  cannot 
be reduced (at least for $n=2^k$, $k\geqslant 1$), because it is known \cite{peterson} 
that the real projective plane $\mathbf{P}^n$ cannot be embedded in 
$\mathbf{R}^{2n-1}$ up to homotopy type for such~$n$.

In shape theory,  the following natural question arises (it was asked by Borsuk 
in~\cite{borsuk-rus}).

\begin{problem}\label{prob-1}
Is it true that any $n$-dimensional metrizable continuum is shape embedded in 
$\mathbf{R}^{2n}$? 
\end{problem}

Clearly, the answer is ``no'' already for $n=1$, because solenoids, which are 
one-dimensional continua, cannot be shape embedded in $\mathbf{R}^2$. Indeed, solenoids are 
not movable, while all compact sets in the plane are movable. Borsuk \cite{borsuk-rus} 
proved a more general theorem, according to which a one-dimensional metrizable 
continuum $X$ is shape embedded in $\mathbf{R}^2$ if and only if $X$ is movable.

Problem~\ref{prob-1} also has a negative solution in the case of metrizable continua $X$ of shape 
dimension $\sd X = n$. Duvall and Husch \cite{duvall-husch-embedding}  
constructed a metrizable continuum $X$ of shape dimension $\sd X = n$, $n=2^k$, which is not shape 
embedded in $\mathbf{R}^{2n}$ for each $k>1$. Nevertheless, under certain additional 
constraints, Problem~\ref{prob-1} has a  positive solution. The first results 
in this direction were obtained by Ivan\v{s}i\'{c} \cite{ivansic-embedding}.

\begin{theorem}
Let $X$ be a pointed 1-movable metrizable continuum with $\sd X =n \geqslant 3$. Then $X$ is shape 
embedded in $\mathbf{R}^{2n}$. 
\end{theorem}

This theorem is a corollary of the following more general theorem (see Husch and Ivan\v{s}i\'{c}'s 
paper~\cite{hus-ivan-embedding}).

\begin{theorem}
Let $X$ be a shape $r$-connected $(r+1)$-movable metrizable continuum with $\sd X =n \geqslant 
3$. Then $X$ is shape embedded in $\mathbf{R}^{2n-r}$. 
\end{theorem}

In the last theorem, the $(r+1)$-movability assumption can be removed 
at the expense of increasing the dimension of $\mathbf{R}^{2n-r}$ by 1 (Husch and 
Ivan\v{s}i\'{c} \cite{hus-ivan-embedding}).

\begin{theorem}
Let $X$ be a shape $r$-connected metrizable continuum with $\sd X =n$ and $n-r \geqslant 2$. Then 
$X$ is shape embedded in $\mathbf{R}^{2n-r+1}$. 
\end{theorem}

Of interest is also the following problem, which was posed by Borsuk in~\cite{borsuk-rus}.

\begin{problem}\label{prob-2}
Let $X$ and $Y$ be compact metrizable spaces such that $\sh X \leqslant \sh Y$ and $Y\subset 
\mathbf{R}^n$. Is it true that $X$ is shape embedded in $\mathbf{R}^n$? 
\end{problem}

Husch and Ivan\v{s}i\'{c} \cite{hus-ivan-shape} obtained a positive solution of this problem 
under fairly strong constraints.

\begin{theorem}
Let $X$ and $Y$ be metrizable continua such that $Y \subset \mathbf{R}^n$, $\dim Y = k$, $X$ 
has the shape of a finite complex, and $\sh X \leqslant \sh Y$. If $3k < 2(n-1)$ and $k \geqslant 
3$, then $X$ is shape embedded in $\mathbf{R}^n$. 
\end{theorem}

An example constructed by Kadlof \cite{kadlof-anexample} shows that Problem~\ref{prob-2} has 
a negative answer. To be more precise, there exists a compact connected two-dimensional CW-complex 
$X$ which is homotopy dominated by a compact connected polyhedron $Y \subset \mathbf{R}^3$ 
but is not shape embedded in~$\mathbf{R}^3$.

Recall that a space $X$ is said to be \emph{$\Pi$-similar}, where $\Pi$ is a family of finite 
polyhedra, if  each open finite  cover of $X$ has an open refinement whose nerve is homeomorphic 
to a polyhedron $P$ in $\Pi$. If $\Pi=\{T^n, n\geqslant 0\}$, where each $T^n = S^1 \times \ldots 
\times S^1$ is the $n$-torus, then $X$ is said to be \emph{$T^n$-similar}. Any 
$\Pi$-similar continuum is the limit of an inverse system consisting of elements of the family 
$\Pi$. Questions concerning  shape embeddings of $T^n$-similar continua in Euclidean spaces were 
studied by Keesling and Wilson \cite{keesling-wilson}. They proved the following theorems.

\begin{theorem}
An arbitrary $n$-dimensional compact connected Abelian topological group $A$ is (topologically) 
embedded in $\mathbf{R}^{n+2}$. 
\end{theorem}

\begin{theorem}
Let $X$ be a $T^n$-similar continuum. Then $X$ is shape embedded in $\mathbf{R}^{n+2}$.
\end{theorem}

The proofs of these theorems essentially use McCord's theorem 
\cite{mccord-embedding} on an embedding of the limit of an inverse system in $\mathbf{R}^n$ and 
a theorem of Keesling \cite{keesling-7} (see Theorem~\ref{th-kess}) characterizing 
$T^n$-similar continua.

In \cite{keesling-wilson} is was also proved that any $T^n$-similar continuum is embedded in 
$\mathbf{R}^{2n}$. For $I^n$-similar continua, this was proved by Isbell~\cite{isbell-embeddings}.

\section{Cell-Like Maps and Shape Theory}

Cell-like maps (or, briefly, CE-maps) play a very important role in geometric topology. They were 
defined by Armentrout \cite{armentrout-1} in terms of $UV$-properties and have already been 
extensively studied when shape theory was discovered. Here we discuss those aspects of cell-like 
maps which are of particular interest for shape theory. Information about other questions 
can be found in the fairly extensive literature on CE-maps, which includes, in particular, the 
survey papers \cite{lacher-2} by Lacher, \cite{edwardsR.} by Edwards, 
\cite{geoghegan-1} by Geoghegan, \cite{rubin-1} by Rubin, and \cite{dra-schepin} by Dranishnikov 
and Shchepin.

\begin{definition}
A map $f\colon X\to Y$ is said to be {\it cell-like} if all fibers $f^{-1}(y)$, $y\in Y$, 
have trivial shape, i.e., $\sh(f^{-1}(y))=0$. 
\end{definition}

If $X$, $Y$, and $Z$ are metrizable ANRs and $f\colon X\to Y$ and $g \colon  Y\to Z$ 
are cell-like maps, then the composition $g\circ f \colon  X \to Z$ is cell-like as well. 
Therefore, the compact metrizable ANRs and cell-like maps form a category. This is not the 
case for compact sets not being ANRs (see Taylor's paper~\cite{taylor}).

The following version of Smale's theorem in shape theory was proved by Dydak \cite{dydak} (see also 
papers  \cite{kuper-2} by Kuperberg,  \cite{kodama-2} by Kodama, and \cite{bogat} by Bogatyi).

\begin{theorem}\label{th-dyd-sm}
Let $X$ and $Y$ be compact metrizable sets, and let $f\colon X\to Y$ be a cell-like map. Then 
the induced homomorphisms  
$$
\pro-\pi_n(f)\colon \pro-\pi_n(X,*)\to  \pro -\pi_n(Y,*)
$$ 
of homotopy pro-groups are isomorphisms for all $n$. 
\end{theorem}

The following theorem is a homotopy characterization of cell-like maps.

\begin{theorem} (Lacher \cite{lacher-2}).
A map $f \colon  X \to Y$ between compact ANRs $X$ and $Y$ is cell-like if and only if, given 
any open set $V\subset Y$, the map $f|f^{-1}(V) \colon  f^{-1}(V) \to V$ is a homotopy equivalence. 
\end{theorem}

This theorem directly implies the following assertion. 

\begin{theorem}\label{th-6}
Any cell-like map $f\colon X\to Y$ between compact ANRs $X$ and $Y$ is a homotopy equivalence.
\end{theorem}

However, cell-like maps between arbitrary compact spaces (not necessarily being 
absolute neighborhood retracts) do not have this property, they are not necessarily shape 
equivalences. Taylor \cite{taylor} constructed a cell-like map 
from the Kahn compact space $X$ \cite{kahn}  to the Hilbert cube $Q$ which not is a shape 
equivalence, because the shape of the Kahn compact space $X$ is not trivial. Taylor's construction 
was based on Adams' result in~\cite{adams}. Keesling \cite{keesling-3} used Taylor's example 
to construct a cell-like map of the Hilbert cube $Q$ to a non-movable compact metrizable space $Y$.  
Koslowski and Segal \cite{kozlowski-segal-0} and Dydak \cite{dydak} constructed examples 
of cell-like maps between movable compact metrizable spaces which are not shape 
equivalences either. Developing Keesling's idea \cite{keesling-3}, van~Mill \cite{vanmill}  
constructed a cell-like map $f\colon Q\to X$, where $X$ is not movable and all fibers 
$f^{-1}(x)$, $x\in X$, are homeomorphic to~$Q$.

Nevertheless, in some cases, cell-like maps between compact metrizable spaces are shape equivalences 
(see papers \cite{sher} by Sher,  \cite{bogat} and \cite{bogat-1} by Bogatyi, and 
\cite{kodama-2} by Kodama).

\begin{theorem}\label{th-7}
A cell-like map $f\colon X\to Y$ between finite-dimensional compact metrizable spaces $X$ and $Y$  
is a shape equivalence. 
\end{theorem}

Indeed, by virtue of Smale's theorem~\ref{th-dyd-sm}, the map $f$ induces isomorphisms of all 
homotopy pro-groups, and since $\dim X, \dim Y < \infty$, it follows by Whitehead's 
theorem~\ref{th-mor} that $f$ is a shape equivalence. Both the Kahn compact space and the Hilbert 
cube are infinite-dimensional; therefore,  Taylor's example mentioned above shows that the  
finite-dimensionality assumption on the compact metrizable  spaces in the last theorem  is 
essential.

Thus, if $X$ and $Y$ have finite dimension or are ANRs, then any cell-like map $f \colon  X \to 
Y$ is a shape equivalence.

An important strengthening of the notion of a cell-like map is the notion of a \emph{hereditary 
shape equivalence}, which was introduced by Kozlowski in \cite{kozlowski-1} and~\cite{kozlowski-2}.

\begin{definition}
A surjective map $f \colon  X \to Y$ between compact metrizable spaces is called a \emph{hereditary 
shape equivalence} if, for any closed set $B\subset Y$, the map $f|f^{-1}(B) \colon  f^{-1}(B)  \to 
B$ is a shape equivalence. 
\end{definition}

Obviously, any hereditary shape equivalence $f \colon  X \to Y$ is a cell-like map. However, 
the converse is false: the map of the Kahn compact space on the Hilbert cube constructed by Taylor 
in \cite{taylor} is cell-like, but it is not a shape equivalence.

Nevertheless, there exists a large class of cell-like maps being hereditary shape 
equivalences. The following two theorems are due to Koslowski \cite{kozlowski-1} (see 
also~\cite{dydak-segal-1}).

\begin{theorem}\label{th-10}
Any cell-like map $f\colon X\to Y$ between compact ANRs is a hereditary shape equivalence.
\end{theorem}

\begin{theorem}\label{th-11}
Let $f\colon X\to Y$ be a cell-like map between compact metrizable spaces. If $\dim Y < \infty$,  
then $f$ is a hereditary shape equivalence. 
\end{theorem}

Kozlowski \cite{kozlowski-1} proved that hereditary shape equivalences have the following 
important property.

\begin{theorem}\label{th-9}
Let $f \colon  X \to Y$ be a hereditary shape equivalence of compact metrizable spaces. If $X$ is 
an ANR, then so is $Y$. 
\end{theorem}

In this theorem the condition that the  map $f \colon  X \to Y$ is a hereditary shape equivalence 
cannot be replaced by the weaker condition that this is a cell-like map.  Indeed, the cell-like 
map constructed by Keesling in \cite{keesling-3} maps the Hilbert cube $Q$ to a non-movable 
compact metrizable space $Y$ which is not an ANR.

Thus, the image of a compact ANR under a cell-like map is not necessarily a compact ANR.

The following theorem of West \cite{west-2} shows that all compact metrizable ANRs are images 
of $Q$-manifolds under cell-like maps.

\begin{theorem}
For any metrizable compact ANR $X$, there exists a $Q$-manifold $M$ and a cell-like map 
$f\colon M\to X$. 
\end{theorem}

This theorem solves Borsuk's problem \cite{borsuk-9} of the finite-dimensionality 
of the homotopy type of a compact ANR, because any $Q$-manifold has the homotopy type of a finite 
CW-complex (Chapman \cite{chapman-3}) and cell-like maps between compact ANRs are 
homotopy equivalences (see Theorem~\ref{th-6}).

As we have already mentioned in relation to Theorem~\ref{th-9}, the image of a compact ANR 
under a cell-like map is not necessarily a compact ANR. The following question 
arises: \emph{Is it true that the image of any finite-dimensional compact ANR under a cell-like 
map is a compact ANR?} It follows from results of Koslowski \cite{kozlowski-1} that this 
question is equivalent to the following one: \emph{Is it true that the image 
of any finite-dimensional compact ANR under a cell-like map is finite-dimensional?}

All these questions are closely related to the 
\emph{dimension-raising problem for cell-like maps}:

(i) Is it true that the image of any finite-dimensional compact metrizable space under a cell-like 
map is a finite-dimensional compact space?

In this connection we mention that the cell-like maps of compact metrizable spaces do not increase 
cohomological dimension, and maps of compact metrizable spaces which are hereditary shape 
equivalences do not increase the dimension $\dim$ (Kozlowski~\cite{kozlowski-1}).

In shape theory, problem (i) can also be formulated in the following equivalent form (see 
\cite{kozlowski-1}): \emph{Let $X$ be a finite-dimensional compact metrizable space, and let 
$f\colon X\to Y$ be a cell-like map. Is it true that $f$ is a shape equivalence?}

Problem (i) of raising dimension by cell-like maps has turned out to be difficult and 
instructive. Many topologists have long been taking efforts to solve it. 
Edwards \cite{edwardsR.-0} and Walsh \cite{walsh} proved that a negative answer to 
question (i) is equivalent to a positive solution of the following well-known problem of 
Aleksandrov \cite{aleks-01}, \cite{aleks-02}:

(ii) Does there exist an infinite-dimensional compact metrizable space of finite cohomological 
dimension?

Aleksandrov posed this problem after proving the following fundamental theorem of 
homological dimension theory.

\begin{theorem}[Aleksandrov \cite{aleks-01}]
\label{last theorem}
The dimension of any finite-dimensional compact metri\-zable space coincides with its cohomological dimension. 
\end{theorem}

Problem (ii) essentially reduces to that of the validity of Theorem~\ref{last theorem}
without the assumption that the given  compact metrizable space is finite-dimensional. In 1988 
Dranishnikov \cite{dranish-1}, \cite{dranish-2} solved Aleksandrov's problem by constructing 
an example of an infinite-dimensional compact space $X$ having finite cohomological dimension 
$c$-$\dim_{\mathbb{Z}} X \leqslant 3$.

%------------------------------------------------------------------

\end{document}